\documentclass[article,12pt]{amsart}
\pdfoutput=1
\usepackage[utf8]{inputenc}
\usepackage{amsmath,amssymb,amsfonts,stmaryrd,amsthm,enumerate,graphicx}

\usepackage[backend=bibtex, style=numeric-comp,firstinits=true]{biblatex}

\bibliography{references.bib}
\usepackage[T1]{fontenc}
\usepackage{lineno,hyperref}
\usepackage{graphicx}
\usepackage{sidecap}
\usepackage{xspace}
\usepackage{pifont}
\usepackage{cancel}
\usepackage{soul}
\newcommand{\vertiii}[1]{{\left\vert\kern-0.25ex\left\vert\kern-0.25ex\left\vert #1 
    \right\vert\kern-0.25ex\right\vert\kern-0.25ex\right\vert}}
\newtheorem{theorem}{Theorem}[section]
\newtheorem{lem}[theorem]{Lemma}
\newtheorem{algorithm}[theorem]{Algorithm}

\usepackage{xcolor}
\usepackage{booktabs}
\usepackage{pgfplots}
\usepackage{pgfplotstable}
\usepackage{tikz}
\usepackage{tkz-euclide}
\usetikzlibrary{decorations.pathreplacing,shapes,arrows,shapes.misc,calc,intersections,arrows,external}
\usepgfplotslibrary{external}
\modulolinenumbers[5]

\usepackage{geometry}
\begin{document}
\author[S. Giani]{Stefano Giani}
\address[]{Durham University, Stockton Road, Durham, DH1 3LE UK}
\email{stefano.giani@durham.ac.uk}
 
\author[L. Grubišić]{Luka Grubišić}
\address[]{University of Zagreb 
Bijenička 30 
10000 Zagreb, HR}
\email{luka.grubisic@math.hr}

\author[L. Heltai]{Luca Heltai}
\address[]{Scuola Internazionale Superiore di Studi Avanzati, Via Bonomea 265, 34136 Trieste, IT} 
\email{luca.heltai@sissa.it}

\author[O. Mulita]{Ornela Mulita}
\address[]{Scuola Internazionale Superiore di Studi Avanzati, Via Bonomea 265, 34136 Trieste, IT} 
 \email{omulita@sissa.it}
 
\title[SA-PINVIT for elliptic eigenvalue problems]{Smoothed-adaptive perturbed inverse iteration for elliptic eigenvalue problems}

\begin{abstract}
We present a perturbed subspace iteration algorithm to approximate the lowermost eigenvalue cluster of an elliptic eigenvalue problem. As a prototype, we consider the Laplace eigenvalue problem posed in a polygonal domain. The algorithm is motivated by the analysis of inexact (perturbed) inverse iteration algorithms in numerical linear algebra. We couple the perturbed inverse iteration approach with mesh refinement strategy based on residual estimators. We demonstrate our approach on model problems in two and three dimensions.
\end{abstract}
\keywords{\footnotesize{elliptic eigenvalue problem,  inexact perturbed inverse iteration, mesh adaptation, mesh construction, Laplace operator, inexact solve.}}

\maketitle

\section{Introduction}\label{sec:introduction}
\subsection{Motivation and literature}\label{subsec:motivation_literature} 

We present an eigensolver designed for adaptive finite element methods (AFEM), which increases the efficiency of classical algorithms when applied to a sequence of locally refined meshes, by carefully reducing the accuracy at which eigenproblems in intermediate levels are solved. As an a posteriori error estimator we use a standard residual error estimator such as those from \cite{carstensen2011oscillation}. The computational costs of adaptive eigenvalue solvers are mostly dominated by the cost of the iterative algebraic eigenvalue solver, and any method which can generate a sequence of meshes by saving accurate solves in all loops, while maintaining the overall accuracy, has a potential for increased efficiency.

 The method which we propose is motivated by the work on inexact inverse iteration solvers in numerical linear algebra \cite{Oliveira,Saad16} and recent work in the perturbed iterative methods for source problems \cite{mulita2019quasi}. In the finite element community, perturbed eigenvalue solvers have been used in the context of numerical homogenization and two level methods. There the aim is to try to achieve fine level accuracy more efficiently by using coarse level solves as a form of preconditioning \cite{malqvist2012computation,Xu2001}. Also, for the work on convergence of preconditioned eigenvalue solvers for uniformly refined meshes see \cite{neymeyr2006geometric,knyazev2001toward}.

Eigenvalue problems typically require the solution -- as intermediate steps -- of highly ill conditioned source problems. Here we mention that even a small algebraic residual does not guarantee a good accuracy of the resulting solution, neither for linear systems nor for eigenvalue problems \cite{mikedlar2015story, hestenes1952methods,arioli2004stopping}. On the other hand, solving the linear algebraic problems to a (much) higher accuracy than the order of the discretization error not only does not improve the overall accuracy but also significantly increases the computational cost \cite{gockenbach2006understanding}. These reasons make the study of the algebraic error an integral part of the adaptive FEM.  

Historically, the majority of the AFEM publications has considered exact solutions of the algebraic problems. However, recent developments of many authors dedicate a great deal of effort to account for inexactness of the algebraic approximations and introduce stopping criteria based on the interplay between discretization and algebraic computation in adaptive FEM. Among others, we mention the seminal contributions for boundary value problems \cite{becker1995adaptive, arioli2004stopping, jiranek2010posteriori, arioli2013stopping, arioli2013interplay, ern2013adaptive, papevz2018estimating, miracci2020multilevel, daniel2020guaranteed, MALLIK2020112367, DANIEL2020112607}.

This becomes a much more complicated task for eigenvalue problems which, by their nature, are nonlinear. Incorporating the algebraic error in adaptive eigenvalue solvers is no new development; see \cite{mehrmann2011adaptive,mehrmann2011nonlinear,neymeyr2002posteriori,mikedlar2011inexact}. When dealing with inexact AFEM, issues such as convergence and optimality are of even greater interest; see  \cite{giani2009convergent,garau2009convergence,carstensen2011oscillation,dahmen2008adaptive,dai2008convergence,garau2011convergence}. For an adaptive finite element method with asymptotic saturation for eigenvalue problems see \cite{carstensen2014adaptive}.

The preconditioned inverse iteration \cite{knyazev2003geometric, neymeyr2001geometric, neymeyr2001geometricc, neymeyr2006geometric} is a well-established iterative method
which admits quasi-optimal computational complexity on uniform meshes. The perturbed preconditioned inverse
iteration uses approximation application operators. Its convergence was proved in \cite{binev2004adaptive}, where the bounds for the convergence rate depend on the eigenvalue gap and the quality of the preconditioner. In \cite{rohwedder2011perturbed} the authors consider a perturbed preconditioned inverse
iteration for operator eigenvalue problems with applications to adaptive wavelet discretization, by  exploiting the theory of best $N$-term approximation to prove optimality of AFEM. 

In \cite{carstensen2012adaptive} the authors present the first adaptive finite element eigenvalue solver (AFEMES) of overall asymptotic quasi-optimal complexity for both exact and inexact algebraic approximations, i.e., for sufficiently small mesh-sizes the error is optimal up to a generic multiplicative constant. In particular, under the assumption that the iteration error for two consecutive AFEM steps is small in comparison with the size of the residual a posteriori error estimate, they prove quasi optimality of the inexact inverse iteration coupled with adaptive finite element method for the class of selfadjoint elliptic eigenvalue problems. A similar analysis of convergence and a quasi-optimality of the inexact inverse iteration coupled with adaptive finite element methods was presented in \cite{zeiser2010optimality} for operator
eigenvalue problems.

Another important result in this direction is provided in \cite{carstensen2014guaranteed}, where an adaptive algorithm which monitors the discretisation error, the maximal mesh-size, and the algebraic eigenvalue error is presented. The authors prove fully computable two-sided bounds on the eigenvalues of the Laplace operator on arbitrarily coarse meshes and demonstrate the reliability of the guaranteed error control even with inexact solve of the algebraic eigenvalue problem.

For related results on non-selfadjoint elliptic eigenvalue problems that account for algebraic inexactness we refer to \cite{becker2001optimal,carstensen2011adaptive,gedicke2014posteriori, meidner2009goal, mikedlar2011inexact, rannacher2010adaptive}. 

By following a completely different approach, in \cite{mulita2019quasi, mulita2019thesis}, the authors introduce the smoothed adaptive FEM (S-AFEM), which improves the computational efficiency by mimicking the ascending phase of v-cycle multigrid methods. This highly interesting and promising fast solver is a novel idea that has been derived in the context of selfadjoint elliptic PDEs and has never been used for eigenvalue problems. The reason behind the success of this strategy is that classical residual-based a posteriori error estimators are not sensitive to low frequencies in the solution. Consequently, their application to very inaccurate approximate solutions in intermediate loops –only capturing high frequency oscillations through a smoother– produces an equally good grid refinement pattern in each loop, at a fraction of the computational cost. 

In the spirit of \cite{mulita2019quasi}, the eigenvalue solver that we propose can be seen as an extension of  the strategy of S-AFEM in the context of elliptic eigenvalue PDEs. 

\subsection{Outline}\label{subsec:outline} This paper is organized as follows. In Sect.~\ref{sec:methodology} we present the model problem together with its conforming FEM approximation, and briefly discuss the main ingredients of AFEM.  Sect.~\ref{sec:S-AFEMES} presents and analyses smoothed adaptive perturbed inverse iteration (SA-PINVIT), after briefly describing PINVIT. 
In Sect.~\ref{sec:numerical_experiments} we present numerical experiments that validate our strategy and finalize with some concluding remarks in Sect.~\ref{sec:summary}.

\section{Model problem and approximation}\label{sec:methodology}
Let us consider the symmetric Laplace eigenvalue problem: Seek a non-trivial eigenpair $(\lambda, u) \in \mathbb{R}\times  H^1_0(\Omega; \mathbb{R})$ such that
\begin{equation}\label{eq:laplace_model_problem}
    -\Delta u = \lambda u \quad\text{in $\Omega$} \quad\text{and} \quad u=0 \quad\text{on $\partial \Omega$},
\end{equation}
where $\Omega \subsetneq \mathbb{R}^d, \,d = 1,2,3$ is a bounded, connected polyhedral Lipschitz domain and $\partial \Omega$ is its boundary.
This simple -- but significant -- model problem will serve as a prototype elliptic, second-order, selfadjoint partial differential operator with a compact resolvent. We refer to the survey article~\cite{grebenkov2013geometrical} and to the references therein.

It is well known (see e.g., \cite{babuvska1991eigenvalue}) that problem~\eqref{eq:laplace_model_problem}
has countably many positive eigenvalues which do not have a finite accumulation point. The eigenvalues will be ordered increasingly
\begin{equation}
    0 < \lambda_1 < \lambda_2 \le \lambda_3 \le \dots
\end{equation}
where we count them according to their multiplicity. Furthermore, there exists an orthonormal basis $(u_1, u_2, u_3, \dots)$ of corresponding eigenvectors. Let us note that, in the case in which $\Omega$ is a connected domain, the eigenvalue $\lambda_1$ is a simple eigenvalue and one can choose an associated eigenvector $u_1$ as a positive function. In this paper we will consider both the case of a simple eigenvalue $\lambda_1$ as well as the case of a cluster of $r\in\mathbb{N}$ lowermost eigenvalues -- counting according to multiplicity -- which are separated by a strictly positive distance from the unwanted component of the spectrum. In Figure~1 
\begin{figure}[hbt!]
  \centering
  \includegraphics[width=.9\linewidth]{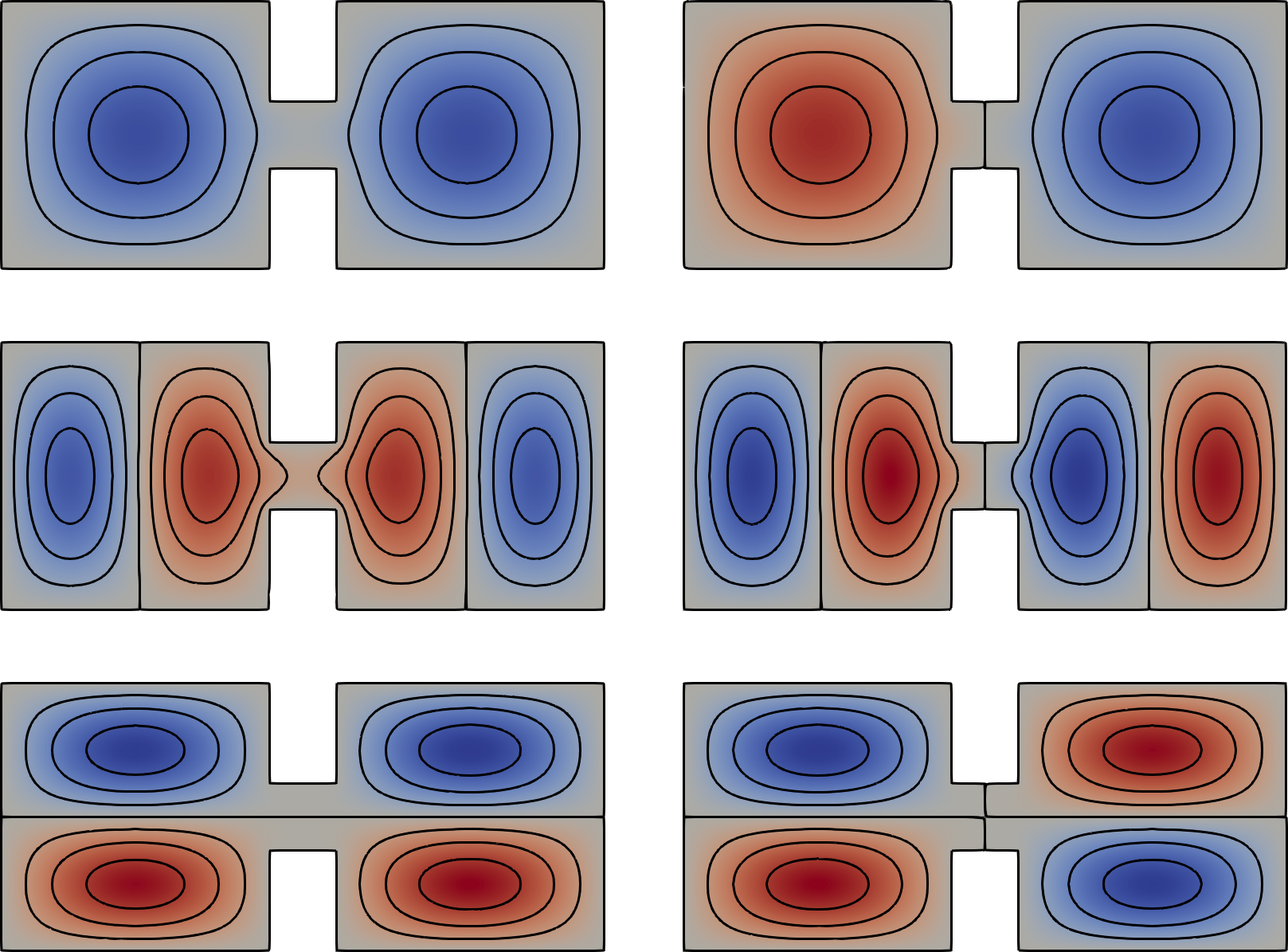}
  \caption{\footnotesize{Six lowermost eigenfunctions of the homogeneous Laplace problem on the dumbbell domain in 2D.}} \label{fig:eigenfunctions_dumbbell_mesh}
\end{figure}
we show the six lowermost eigenfunctions of the homogeneous Laplace problem on the dumbbell domain in 2D, that we will discuss in Sect.~\ref{subsec:dumbbell_2d}.

The weak problem seeks for a non-trivial eigenpair $(\lambda,u)\in \mathbb{R}\times \{V:=H^1_0(\Omega)\}$ with $b(u,u)=1$ and

\begin{equation}\label{eq:weak_laplace_model_problem}
    a(u,v)=\lambda b(u,v) \quad \text{for all $v \in V$.}
\end{equation}
The bilinear forms $a:V\times V \rightarrow \mathbb{R}$ and  $b: H\times H \rightarrow \mathbb{R}$ are defined by
\begin{equation}
    a(u,v):= \int_{\Omega} \nabla u \cdot \nabla v\, dx \quad \text{and} \quad  b(u,v):= \int_{\Omega} u v\, dx
\end{equation}
and induce norms $\vertiii{\bullet}:= | \bullet |_{H^1(\Omega; \mathbb{R})}$ on $V$ and $\|\bullet \|_{L^2(\Omega;\mathbb{R})}$ on $H:=L^2(\Omega;\mathbb{R})$. 

The well-established adaptive finite element routine (AFEM) \cite{binev2004adaptive,cascon2008quasi,dorfler1996convergent,stevenson2007optimality} computes a sequence of discrete subspaces
\begin{equation}
V_0 \subsetneq V_1\subsetneq V_2\subsetneq \dots \subsetneq V_{\ell} \subsetneq V
\end{equation}
using local refinement of the underlying mesh of the computational domain $\Omega$. The corresponding sequence of meshes consists of nested regular polygonal/polyhedral tessellation   $(\mathcal{T}_{\ell})_{\ell}$ in the sense of Ciarlet \cite{ciarlet2002finite} of the domain $\Omega$. The AFEM can be described as the following loop
\begin{equation}\label{afem_loop}
   \text{SOLVE $\rightarrow$ ESTIMATE $\rightarrow$  MARK $\rightarrow$ REFINE}.
\end{equation}

The phases of the AFEM routine are shortly described in what follows. Let $V_{\ell}:=P_m(\mathcal{T}_{\ell}) \cap V$ denote the finite-dimensional subspace of fixed order $m > 0$,
for the conforming finite element space $P_m(\mathcal{T}_{\ell})$ of polynomials of degree at most $m$ based on $\mathcal{T}_{\ell}$, and let $N_{\ell} :=
\text{dim}(V_{\ell})$. The corresponding discrete eigenvalue problem for each level reads: Seek a non-trivial eigenpair
$(\lambda_{\ell}, u_{\ell}) \in \mathbb{R} \times V_{\ell}$ with $b(u_{\ell}, u_{\ell}) = 1$ and
\begin{equation}\label{eq:discrete_model_problem}
    a(u_{\ell},v_{\ell})=\lambda_{\ell} b(u_{\ell},v_{\ell}) \quad \text{for all $v_{\ell}\in V_{\ell}$}.
\end{equation}

Given a mesh $\mathcal{T}_{\ell}$ on the level $\ell$, step SOLVE computes (assembles) the stiffness matrix $A_{\ell}$ and the mass
matrix $M_{\ell}$ and solves the $N_{\ell}$-dimensional \textit{generalized algebraic eigenvalue problem} associated to discrete problem~\eqref{eq:discrete_model_problem}
\begin{equation}\label{eq:generalized_eigenvalue_problem}
    A_{\ell}\mathbf{u}_{\ell} = \lambda_{\ell} M_{\ell} \mathbf{u}_{\ell}, 
\end{equation}
where 
\begin{equation}
    u_{\ell}=\sum_{i=1}^{N_{\ell}} \mathbf{u}_{\ell,i}\varphi^{(i)}_{\ell}, \quad
    A_{\ell}:= [a(\varphi^{(i)}_{\ell}, \varphi^{(j)}_{\ell})]_{1 \le i,j,\le N_{\ell}}, \quad M_{\ell}:= [b(\varphi^{(i)}_{\ell}, \varphi^{(j)}_{\ell})]_{1 \le i,j,\le N_{\ell}}
\end{equation}  
for $V_{\ell}=\text{span}\{\varphi^{(1)}_{\ell}, \dots, \varphi_{\ell}^{({N}_{\ell})}\}.$
In practice, these discrete eigenvalue problems are solved \textit{inexactly} using iterative algebraic eigenvalue solvers. Well-established iterative eigenvalue solvers that satisfy the convergence and complexity assumptions are the preconditioned inverse iteration (PINVIT) \cite{knyazev2003geometric} or the locally optimal (block) preconditioned conjugate gradient (LOBPCG) \cite{knyazev2001toward}. In this work we limit ourselves to use PINVIT itarations, and its application to adaptive eigenproblems. We refer to classical locally adaptive PINVIT with A-PINVIT, which consists on the applications of several loops of the type:
\begin{equation}\label{A-PINVIT_loop}
   \text{PINVIT $\rightarrow$ ESTIMATE $\rightarrow$  MARK $\rightarrow$ REFINE}.
\end{equation}

The final computation, in general, involves  the discretisation error as well as the algebraic error (in the eigenfunction and eigenvalue of interest) stemming from the termination of the iterative algebraic eigenvalue solver. In practice, the computational costs for the iterative algebraic eigenvalue solver dominate the overall computational costs.

\subsection{A residual error estimator for Galerkin approximations}
Step ESTIMATE uses a posteriori error estimators, which are computable quantities defined in terms of the discrete approximation that can estimate the actual error in a suitable norm. First, we review the results for finite element approximations which satisfy Galerkin optimality. We refer to  \cite{carstensen2004some} for the formal definition and for an in-depth description. We consider explicit residual-based a posteriori error estimators \cite{dai2008convergence, duran2003posteriori,garau2009convergence,giani2009convergent}.
Let $p_{\ell}$ denote the discrete gradient and $\mathcal{E}_{\ell}$ denote the set of inner edges ($d=2$) or inner faces $(d=3)$ of elements of $\mathcal{T}_{\ell}$. For $E \in \mathcal{E}_{\ell}$ let $T_{+}, T_{-} \in \mathcal{T}_{\ell}$ be two neighbouring elements such that $E=T_{+} \cap T_{-}$. The jump of the discrete gradient $p_{\ell}$ along an inner edge $E \in \mathcal{E}_{\ell}$ in normal direction $\nu_{E}$, pointing from $T_{+}$ to $T_{-}$, is defined by $[p_{\ell}] \cdot \nu_{E}:= (p_{\ell}\mid_{T_{+}} - p_{\ell}\mid_{T_{-}}) \cdot \nu_{E}.$
 The error in the eigenfunction or eigenvalue is estimated based on the solution $(\lambda_{\ell}, u_{\ell})$ of the underlying algebraic eigenvalue problem via the explicit residual-based a posteriori error estimator defined by

\begin{equation}\label{eq:residual_based_estimator}
    \begin{aligned}
    \eta^2(\lambda_{\ell}, u_{\ell}) &:= \sum_{T \in \mathcal{T}_{\ell} } \eta_{T}(\lambda_{\ell},u_{\ell})^2
    \end{aligned}
\end{equation}
with $d=2,3$ and
\begin{equation}
    \begin{aligned}
    J^2_E(\lambda_{\ell},u_{\ell})&:=|E|^{1/(d-1)} \|[p_{\ell}] \cdot \nu_E\|^2_{L^2(E)} \\
    \eta^2_{T}(\lambda_{\ell},u_{\ell}) & := |T|^{2/d}
     \| \lambda_{\ell} u_{\ell} + div(p_{\ell})\|^2_{L^2(T)} + \sum_{E \in \mathcal{E}_{\ell}} J^2_E(\lambda_{\ell},u_{\ell}).
    \end{aligned}
\end{equation}

Based on the local estimators, elements are marked for refinement in a bulk criterion \cite{dorfler1996convergent} such that $\mathcal{M}_{\ell} \subseteq \mathcal{T}_{\ell} \cup \mathcal{E}_{\ell}$ is an (almost) minimal set of marked edges with
\begin{equation}
\begin{aligned}
    \theta \eta^2(\lambda_{\ell}, u_{\ell}) &\le \eta^2 (\lambda_{\ell}, u_{\ell}; \mathcal{M}_{\ell}),\\
    \eta^2(\lambda_{\ell}, u_{\ell}; \mathcal{M}_{\ell})& := \sum_{T \in \mathcal{M}_{\ell} \cap \mathcal{T}_{\ell}} \eta^2_T (\lambda_{\ell}, u_{\ell})
\end{aligned}
\end{equation}
for a bulk parameter $0 <\theta \le 1$. Finally, the mesh is refined locally (REFINE) according to the set $\mathcal{M}_{\ell}$ of marked elements (MARK). 

\subsection{A posteriori error estimates for inexact eigenvector approximation}\label{res_sec}
Let us now briefly review the a posteriori error estimation techniques and results which we will use to argue the reliability of our approach to quasi optimal adaptive mesh construction. If cost of the estimation procedure were not of concern we would have used an auxiliary subspace error estimates such as those utilized in \cite{BGO,GrubisicOvall}. The important feature of these estimators is that for eigenvector approximations which satisfy the Galerkin orthogonality condition for the finite element space we can prove the reliability and the asymptotic exactness of the a posteriori error estimator for the lowermost cluster of eigenvalues.

In the case in which we do not assume that the approximate eigenvectors satisfy the Galerkin orthogonality, we have we have to use the more expensive technique based on duality residual estimates such as those used in \cite{GGO1}. Such estimators, although reliable and asymptotically robust, are to expensive to be used in the context of this paper.

To obtain an a posteriori error estimator which is relatively lightweight and sufficiently robust we will combine the eigenvalue estimator from \cite{GGMO} with the analysis and the estimator from \cite{StrakosPapez}. The analysis from \cite{GGMO} allows us to reduce the eigenvalue a posteriori error estimation to the error estimation for the source problem.

Let us recall the result from \cite{StrakosPapez}. Let the source problem
\begin{equation}\label{source}
-\triangle u = f,\qquad u\in H^{1}_0(\Omega; \mathbb{R})    
\end{equation}
be given and let $V_\ell\subset H^1_0(\Omega;\mathbb{R})$ be a finite element space (piece-wise polynomial) defined by the triangulation $\mathcal{T}_{\ell}$. We will use $u(f)$ to denote the solution of the problem \eqref{source}. It has been established that there exists a constants $C_{int}$ and $C_1$ such for $f\in V_\ell$ and $v_\ell\in V_\ell$
\begin{equation}\label{source2}
\|\nabla(u(f)-v_\ell)\|^2\leq C_1^2 \sum_{E \in \mathcal{M}_{\ell} \cap \mathcal{E}_{\ell}} |E|^{1/(d-1)} \|[\frac{\partial v_{\ell}}{\nu_E}]\|^2_{L^2(E)} + C_{int}^2\|\nabla(u_\ell(f)-v_\ell(f))\|,
\end{equation}
where $[\frac{\partial v_{\ell}}{\nu_E}]$ is the jump of the discrete gradient of $v_{\ell}$ in the direction of the normal of the edge of an element. Note that in the case of a general $f\in\ L^2(\Omega)$, there is an additional data oscillation term in the estimate from \cite{StrakosPapez}, which is not present for $f\in V_\ell$. Further, $u_{\ell}(f)$ will denote the Galerkin approximation of $u(f)$ which satisfies
$$
u_\ell(f)=\mathrm{arg min}\{a(v,v)/2-b(f,v)~:~v\in V_\ell\}~.
$$
Finally we allow for $v_h\in V_h$ to be arbitrary. The constants $C_1$ and $C_{int}$ do not depend on $f$ but do depend on $V_\ell$ and $\mathcal{T}_\ell$. The constant $C_{int}$ is defined as
$$
C_{int}^2=\sup_{f\in L^2(\Omega)}\sup_{v_\ell\in V_\ell}\frac{\|\nabla \mathcal{I}(u(f)-v_\ell)\|}
{\|\nabla(u(f)-v_\ell)\|}~,
$$
where $\mathcal{I}:L^1(\Omega)\to V_h$ is the quasi-interpolation operator for $V_\ell$.
Obviously, $C_{int}$ is quite pessimistic and we will argue that for the type of right hand sides which we chose it can be much smaller. Similar argument was used in \cite[Theorem 2.1]{PAMM_Gru} and we will refine this analysis with a more refined handling of the computable residual estimates. We will combine \cite[Theorem 2.1]{PAMM_Gru} with the analysis from \cite[Theorem 3.8.]{BGO} and \cite[Estimates (35) and (36)]{GGMO} to obtain a computable residual estimate.

Let us introduce some notation. Given linearly independent vectors $v_\ell^{(i)}\in V_\ell$, $i=1, \cdots, r$ we define the space $
\mathcal{V}_\ell=\mathrm{span}\{v_\ell^{(1)},\cdots,v_\ell^{(r)}\}$
and the associated Ritz values 
$$
\tilde{\lambda}_i=\min\{\max\{a(v,v)/b(v,v)~:~v\in U\setminus\{0\}\}~:~U\subset\mathcal{V}_\ell,\;\dim(U)=i\}
$$
and let $\tilde{v}_i\in\mathcal{V}_\ell$ be an linearly independent set of vector such that $\tilde{\lambda}_i=a(\tilde{v}_i,\tilde{v}_i)/b(\tilde{v}_i,\tilde{v}_i)$. The vector $\tilde{v}_i$
are called the Ritz vectors from $\mathcal{V}_\ell$.

\begin{theorem}
Let linearly independent $v_\ell^{(i)}\in V_\ell$, $i=1, \cdots, r$, $r\in\mathbb{N}$ and the associated $\mathcal{V}_\ell$ be a given.
If $\lambda_r<\lambda_{r+1}$, $r=\mathrm{arg min}_{i\in\mathbb{N}}\frac{|\lambda_i-\tilde\lambda_r|}{\sqrt{\lambda_i\tilde{\lambda}_r}}$
and $\lambda_r<\min\{a(v,v)/b(v,v)~:~v\perp\mathcal{V}_\ell\}$, then
$$\aligned
\sum_{i=1}^r\frac{|\lambda_i-\tilde{\lambda}_i|}{\lambda_i}\leq\min\big\{ C_{{\rm clust}}(\mathcal{V}_\ell)
&\sum_{i=1}^r\|\nabla (u(\tilde{v}_i) -\tilde{\lambda_i}^{-1}\tilde{v}_i)\|^2,\\&\quad
r\max_i\|\nabla (u(\tilde{v}_i) -\tilde{\lambda_i}^{-1}\tilde{v}_i)\|\big\}.
\endaligned$$
Here $C_{{\rm clust}}(\mathcal{V}_\ell)=\frac{\lambda_{r+1}+\tilde{\lambda}_i}{|\lambda_{r+1}-\tilde{\lambda}_i|}$ is the spectral gap.
\end{theorem}
\begin{proof}
The proof of this result is a direct combination of \cite[Theorem 2.1]{PAMM_Gru} and \cite[Estimates (35) and (36)]{GGMO} together with the Bauer-Fike residual estimate \cite[Proposition 11]{GGMO}.
\end{proof}
For an analogous finite dimensional version of this theorem we point a reader to \cite[Theorem 4.1]{DrmBosner}. There we can also find a discussion as to when we can relax an apparently stringent restriction on the location of the Ritz values from $\mathcal{V}_\ell$.

The constant $C_{int}$ is a general constant depending on the domain $\Omega$ the subspace $V_\ell$. However, it is to pessimistic in general -- as has been argued in \cite{StrakosPapez} -- and in particular it is so in the case of its application in the analysis of the eigenvalue problem. In the spirit in which the saturation assumption was handled in \cite{PAMM_Gru}, we define the subspace dependent constant 
$$
C_{int}^2(\mathbf{e})=
\sup_{\substack{f\in \mathcal{V}_\ell,\\(f,u)_{L^2}/(u,u)_{L^2}\leq\mathcal{E}}}\sup_{v_\ell\in V_\ell}\frac{\|\nabla \mathcal{I}(u(f)-v_\ell)\|}
{\|\nabla(u(f)-v_\ell)\|}.
$$
Obviously, $C_{int}(\mathcal{V}_\ell)\leq C_{int}$ and so we see that for judiciously constructed subspaces $\mathcal{V}_\ell$ we can achieve much tighter error control by the discrete residual. Such will be the case when choosing $\mathcal{V}_\ell$ as a prolongation of Galerkin eigenvector approximation from a sufficiently fine coarse mesh.

Now, as a direct consequence of the fore mentioned we have the following commutable reliability estimate.
\begin{theorem}\label{sep}
Let linearly independent $v_\ell^{(i)}\in V_\ell$, $i=1, \cdots, r$, $r\in\mathbb{N}$ and the associated $\mathcal{V}_\ell$ be a given.
If $\lambda_r<\lambda_{r+1}$, $r=\mathrm{arg min}_{i\in\mathbb{N}}\frac{|\lambda_i-\tilde\lambda_r|}{\sqrt{\lambda_i\tilde{\lambda}_r}}$
and $\lambda_r<\min\{a(v,v)/b(v,v)~:~v\perp\mathcal{V}_\ell\}$, then
$$\aligned
\sum_{i=1}^r\frac{|\lambda_i-\tilde{\lambda}_i|}{\lambda_i}\leq \min\Big\{&C_{{\rm clust}}(\mathcal{V}_\ell)\big(
C_1^2\sum_{i=1}^r J^2_\ell(\tilde{\lambda}_i^{-1}\tilde{v}^{(i)}) + C_{int}^2\sum_{i=1}^r\|\nabla(u_\ell(\tilde{v}^{(i)})-\tilde{\lambda_i}^{-1}\tilde{v}^{(i)})\|^2\big),\\
&\;\;\sqrt{C_1^2\sum_{i=1}^rJ^2_\ell(\tilde{\lambda}_i^{-1}\tilde{v}^{(i)}) + C_{int}^2\sum_{i=1}^r\|\nabla(u_\ell(\tilde{v}^{(i)})-\tilde{\lambda_i}^{-1}\tilde{v}^{(i)})\|^2}\Big\}~.
\endaligned$$
Where $ J^2_\ell(\tilde{\lambda}_i^{-1}\tilde{v}^{(i)})=\tilde{\lambda_i}^{-1}\sum_{E \in \mathcal{M}_{\ell} \cap \mathcal{E}_{\ell}} |E|^{1/(d-1)} \|[\frac{\partial \tilde{v}^{(i)}}{\nu_E}]\|^2_{L^2(E)}$ is the discrete residual measure and we call $\|\nabla(u_\ell(\tilde{v}^{(i)})-\tilde{\lambda_i}^{-1}\tilde{v}^{(i)})\|^2$ the algebraic error.
For any $\mathbf{e}$, $\mathbf{e}\geq\lambda_r$ we can substitute the constant $C_{int}^2(\mathbf{e})$ for $C_{int}^2$ and the estimate still holds.
\end{theorem}
\begin{proof}
Under the assumptions of the theorem, the Ritz values from the subspace $\mathcal{V}_\ell$ verify
$$\aligned
\sum_{i=1}^r\frac{|\lambda_i-\tilde{\lambda}_i|}{\lambda_i}\leq\min\big\{ C_{{\rm clust}}(\mathcal{V}_\ell)
&\sum_{i=1}^r\|\nabla (u(\tilde{v}_i) -\tilde{\lambda_i}^{-1}\tilde{v}_i)\|^2,\\&\quad
r\max_i\|\nabla (u(\tilde{v}_i) -\tilde{\lambda_i}^{-1}\tilde{v}_i)\|\big\}.
\endaligned$$
We now apply the estimate \eqref{source2} on the associated source problems 
\begin{equation}\label{coarse}
-\triangle u = \tilde{v}_i
\end{equation}
to estimate the error
$$
\|\nabla (u(\tilde{v}_i) -\tilde{\lambda_i}^{-1}\tilde{v}_i)\|, \;\;i=1,\cdots, r.
$$
Here we use $\tilde{\lambda}_i\tilde{v}_i$ as an approximate coarse solution of \eqref{coarse} and directly apply the source problem error estimate \cite[Corollary 4.2]{StrakosPapez}. 
The claim that we can substitute $C_{int}^2(\mathbf{e})$ for $C_{int}^2$ follows directly from \cite[Inequality (4.5)]{StrakosPapez}.
\end{proof}
Let us note that these results imply that our approach is only going to be robust if we deal with the whole cluster of lowermost eigenvalues of interest which are well separated from the rest of the spectrum (in the sense of Theorem \ref{sep}). For eigenvalues high up in the spectrum or for indefinite problems, other techniques and or analysis have to be utilized.

\section{Smoothed-adaptive PINVIT}\label{sec:S-AFEMES}

The algorithm that we propose is inspired by the ascending phase of the v-cycle multigrid methods. Those methods use prolongation to transfer the low frequency information contained in the coarse approximation to a finer –nested–grid, and then apply few steps of a smoothing iteration to improve the accuracy of the solution in the high frequency range. The iteration of this procedure is based on the principle that even a small number of smoothing iterations is sufficient to eliminate the high frequency error, while the prolongation from coarser grids guarantees the convergence in the low frequency regime, providing accurate algebraic solutions in linear computational time \cite{hackbusch2013multi}.

We emphasize that the algebraic solutions for intermediate AFEM loops~\eqref{afem_loop}  serve \textit{solely} to the construction of the final grid, and as a initial guess for the next space in the sequence. Their role is instrumental in triggering the ESTIMATE-MARK-REFINE steps, and in providing a reasonable initial guess for the finest level. 

Smoothed-adaptive PINVIT (SA-PINVIT) is formalized in the following scheme.

\tikzset{
  block/.style = {fill=white, minimum height=3em, minimum width=3em},
  }
 \begin{center}
  \begin{tikzpicture}[auto, node distance=2.5cm,>=latex']
  
    \node [block, name=firstsolve] (firstsolve) {PINVIT};
    \node [block, right of=firstsolve, node distance=2.5cm] (estimate) {ESTIMATE};
    \node [block, right of=estimate,node distance=2.5cm] (mark) {MARK};
    \node [block, right of=mark,node distance=2.3cm] (refine) {REFINE};
    \node [block, below of=refine, node distance=2.3cm] (prolongate){PROLONGATE};
    \node [block, below of=estimate, node distance=2.3cm] (smooth){S-PINVIT};
    \node [block, right of=refine, node distance=2.3cm] (lastsolve) {PINVIT};
    \draw [->] (firstsolve) -- (estimate);
    \draw [->] (estimate) -- (mark);
    \draw [->] (mark) -- (refine);
    \draw [->] (refine) -- (lastsolve);
    \draw [->] (refine) -- (prolongate);
    \draw [->] (prolongate) -- (smooth);
    \draw [->] (smooth)--(estimate);
  \end{tikzpicture}
 \end{center}

Smoothed-PINVIT (S-PINVIT) refers to classical PINVIT iteration, where we allow only a fixed number of iteration steps, and we use a smoother as preconditioner. In this way, we allow for very inaccurate approximate eigenpairs in those intermediate loops by applying a fixed number of few PINVIT iterations, or by setting as intermediate stop criterion a very large error tolerance. The prolongation of the residual from the previous level is used as an initial guess for the next level.

We start by recalling the PINVIT procedure in  Algorithm~\ref{pinvit_alg}, and after some considerations we describe SA-PINVIT in Algorithm~\ref{alg_safem_pinvit}. 

The Rayleigh quotient of any nonzero vector $\mathbf{v}_{\ell}\in \mathbb{R}^{N_{\ell}}$  is defined by  $\mu(\mathbf{v}_{\ell}):=\left \langle A_{\ell} \mathbf{v}_{\ell}, \mathbf{v}_{\ell} \right \rangle/ \left \langle M_{\ell}\mathbf{v}_{\ell}, \mathbf{v}_{\ell} \right \rangle$. PINVIT algorithm can be described as:

\begin{algorithm}[PINVIT($\mathbf{v}_{\ell}^0$, $P_\ell$, \text{max\_iter}, \text{tol})]\label{pinvit_alg} Given the initial guess $\mathbf{v}_{\ell}^0$ with associated Rayleigh quotient $\mu_{\ell}^{0}:=\mu(\mathbf{v}_{\ell}^0)$, generate a sequence of vectors $\{\mathbf{v}_{\ell}^{n} \}_{n \ge0}$ and associated Rayleigh quotients $\{\mu_{\ell}^{n} \}_{n \ge0}$ through steps $1.-3.$, iterated for a maximum of \text{max\_iter} times, or until $(\mu^{n+1}_{\ell}-\mu^{n}_{\ell})/\mu^{n}_{\ell} \leq \text{tol}$:

\begin{enumerate}[1.]
    \item  $\tilde{\mathbf{v}}_{\ell}^{n+1}= \mathbf{v}_{\ell}^n - P_{\ell}^{-1}(A_{\ell} \mathbf{v}_{\ell}^n - \mu_{\ell}(\mathbf{v}_{\ell}^n) M_{\ell} \mathbf{v}_{\ell}^n)$
    \item $\mathbf{v}_{\ell}^{n+1}= \|\tilde{\mathbf{v}}_{\ell}^{n+1}\|^{-1} \tilde{\mathbf{v}}_{\ell}^{n+1}$,
    \item $\mu^{n+1}_{\ell}= \mu_{\ell}(\mathbf{v}^{n+1}_{\ell}).$
\end{enumerate}
\end{algorithm}

In parallel to the inverse iteration for a single vector, we will also consider simultaneous or blocked inverse iteration called BPINVIT.

\begin{algorithm}[BPINVIT($\mathbf{V}_{\ell}^0$, $P_\ell$, \text{max\_iter}, \text{tol})]\label{b_pinvit_alg} Given the initial guess $\mathbf{V}_{\ell}^0\in\mathbb{R}^{N_{\ell}\times r}$ with associated generalized Rayleigh quotient $\Xi_{\ell}^{0}\in\mathbb{R}^{r\times r}$, generate a sequence of matrices $\{\mathbf{V}{_{\ell}}^{n} \}_{n \ge0}$ and associated generalized Rayleigh quotients $\{\Xi_{\ell}^{n} \}_{n \ge0}$ through steps $1.-3.$, iterated for a maximum of \text{max\_iter} times, or until $\|\Xi_{\ell}^{n+1}-\Xi_{\ell}^{n}\|/\|\Xi_{\ell}^{n}\| \leq \text{tol}$:

\begin{enumerate}[1.]
    \item  $\tilde{\mathbf{V}}_{\ell}^{n+1}= \mathbf{V}_{\ell}^n - P_{\ell}^{-1}(A_{\ell} \mathbf{V}_{\ell}^n -  M_{\ell} \mathbf{V}_{\ell}^n\Xi_{\ell}^{n})$
    \item Compute $W$ and $\Xi_{\ell}^{n+1}$ as an eigenvector and eigenvalue matrix of the
    generalized eigenproblem $( (\tilde{\mathbf{V}}_{\ell}^{n+1})^*A_{\ell}\tilde{\mathbf{V}}_{\ell}^{n+1},(\tilde{\mathbf{V}}_{\ell}^{n+1})^*M_{\ell}\tilde{\mathbf{V}}_{\ell}^{n+1})$.
    \item Set $\mathbf{V}_{\ell}^{n+1} = \tilde{\mathbf{V}}_{\ell}^{n+1} W $
\end{enumerate}
\end{algorithm}

Generally, in both Algorithms~\ref{pinvit_alg} and~\ref{b_pinvit_alg}, the initial guesses $\mathbf{V}_{\ell}^0$ and  $\mathbf{V}_{\ell}^0\in\mathbb{R}^{N_{\ell}\times r}$ are chosen randomly, while the preconditioner is a linear operator $P^{-1}_{\ell}$ which is positive definite and spectrally equivalent to $A_{\ell}$. This property can be formulated as: there exists a sufficiently small constant $\gamma_{P_{\ell}} <1$ such that
\begin{equation}\label{eq:spectralequivalence}
    \|Id_{\ell} - P_{\ell}^{-1}A_{\ell}\|_{A_{\ell}} \le \gamma_{P_{\ell}},
\end{equation}
where $\|\bullet\|_{A_{\ell}}$ denotes the operator norm induced
by $A_{\ell}$. An example is given by the wavelet preconditioners as in \cite{rohwedder2011perturbed}. However, this is a stronger property w.r.t. what we need to ensure convergence for our algorithm based on the result from \cite{Saad16}.

For simplicity, let us assume that the preconditioner is spectrally equivalent to $A_{\ell}$. Then the corresponding error propagation equation reads

\begin{equation}\label{eq:error_prop_pinvit}
    \mathbf{v_{\ell}}^{n+1}- \mu_{\ell}(\mathbf{v_{\ell}}^n)A_{{\ell}}^{-1}M_{\ell} \mathbf{v_{\ell}}^{n} = (I-P_{\ell}^{-1}A_{\ell}) (\mathbf{v_{\ell}}^{n} - \mu_{\ell}(\mathbf{v_{\ell}}^n)A_{\ell}^{-1}M_{\ell} \mathbf{v_{\ell}}^{n})
\end{equation}
and it illustrates the dependence between the initial error $\mathbf{v_{\ell}}^{n} - \mu_{\ell}(\mathbf{v_{\ell}}^n)A_{\ell}^{-1}M_{\ell} \mathbf{v_{\ell}}^{n},$  the new iterate  $\mathbf{v_{\ell}}^{n+1}- \mu_{\ell}(\mathbf{v_{\ell}}^n)A_{\ell}^{-1}M_{\ell} \mathbf{v_{\ell}}^{n}$, and the error propagation matrix (reducer) $(I-P_{\ell}^{-1}A_{\ell})$.

Next, we explicitly write the basic PINVIT iteration as 
\begin{equation}
    \begin{aligned}
    \tilde{\mathbf{v}_{\ell}}^{n+1}&= \mathbf{v}_{\ell}^n - P_{\ell}^{-1}(A_{\ell} \mathbf{v_{\ell}}^n - \mu_{\ell}(\mathbf{v}_{\ell}^n) M \mathbf{v}_{\ell}^n)\\
                            &= \mathbf{v}_{\ell}^n - P_{\ell}^{-1}A_{\ell} \mathbf{v}_{\ell}^n + \mu_{\ell}(\mathbf{v}_{\ell}^n)  P_{\ell}^{-1}M_{\ell} \mathbf{v}_{\ell}^n \\
                            &= (Id_{\ell} - P_{\ell}^{-1}A_{\ell}) \mathbf{v}_{\ell}^n + \mu_{\ell}(\mathbf{v}_{\ell}^n)  P_{\ell}^{-1}M_{\ell} \mathbf{v}_{\ell}^n \\
                            &= \mu_{\ell}(\mathbf{v}_{\ell}^n)  P_{\ell}^{-1}M_{\ell} \mathbf{v}_{\ell}^n + \xi_{\gamma_{P_{\ell}}},
    \end{aligned}
\end{equation}
where, using~\eqref{eq:spectralequivalence}, we split the identity into the contracting part and the rest labeled as $\xi_{\gamma_{P_{\ell}}}$. 

We can now analyze the convergence of this scheme using \cite[Proposition 3.7]{Saad16}.

\begin{lem}
Let PINVIT be implemented by Algorithm \ref{pinvit_alg},  let $\gamma_{P_{\ell}}$ be small enough compared to $\lambda_2(A_{\ell})/\lambda_1(A_{\ell})$, and let $ \mathbf{v}_{\ell}^0$ be not orthogonal to $\mathbf{u}_1 (A_{\ell})$. Then 
$$
\tan\angle(\mathbf{v}_{\ell}^n,\mathbf{u}_1 (A_{\ell}))\to 0.
$$
\end{lem}
\begin{proof}
Apply \cite[Proposition 3.7]{Saad16}. Alternatively, see \cite[Theorem 3]{rohwedder2011perturbed}.
\end{proof}
An equivalent result is available for the blocked PINVIT.

\begin{theorem}
Let BPINVIT be implemented by Algorithm \ref{b_pinvit_alg}, let $\gamma_{P_{\ell}}$ 
be small enough compared to $\lambda_{r+1}/\lambda_r$, and let there be no vectors in $\text{Ran}(V_{\ell}^0)$ orthogonal to the subspace $\mathcal{E}_r$ spanned by the eigenvectors belonging to first $r$ eigenvalues. Then 
$$
\tan\angle(\text{Ran}(V_{\ell}^n),\mathcal{E}_r)\to 0.
$$
\end{theorem}

However, a weaker condition than~\eqref{eq:spectralequivalence} still allows us to use the 
convergence analysis from \cite[Proposition 3.7]{Saad16}. Namely, it is sufficient that the preconditioner does a good job only on the iterates it is applied on. Since this is a computation of the dominant (low frequency) eigenspace, a smoother will satisfy such a requirement. 

Quantitatively, this reads
\begin{equation}\label{eq:spectralequivalence_weaker}
\aligned
    \|P_{\ell}^{-1}\mathbf{v} - A_{\ell}^{-1}\mathbf{v}\| &\le \|A_{\ell}^{-1}\|\|A_{\ell} (P_{\ell}^{-1} \mathbf{v} - \mathbf{v})\|\\
    &\le \gamma_{P_{\ell},k} \|\mathbf{v}\|_{A_{\ell}}, \qquad \mathbf{v}\in\text{Ran}(V_{\ell}^k)
\endaligned
\end{equation}
and since we compute the action of $P_{\ell}^{-1}$ by applying several smoothing steps we can control the 
size of $\gamma_{P_{\ell},k}$ by increasing the number of steps applied while monitoring the size
of the discrete residual $\|A_{\ell} (P_{\ell}^{-1} \mathbf{v} - \mathbf{v})\|$.
Condition~\eqref{eq:spectralequivalence_weaker} allows for a smoother, which is going to satisfy this requirement on low frequencies (e.g. those close to $\|A_{\ell}^{-1}\|^{-1}$). Here we have tacitly assumed that the initial subspace had a small  angle with the space spanned by low frequency eigenmodes. Let us point out that even though these results present estimates, their conditions are vary hard to check in practice. The fact that an eigenvalue estimate depends on external information on the unwanted part of the spectrum (e.g. location of $\lambda_2$) is frequently found in eigenvalue bracketing results such as e.g. \cite{Vejchodsk1,Vejchodsk2}. In this context the estimates have to be seen as asymptotic and their main value is indicating the source of instability in the realization of the method.

We now combine these considerations in the next algorithm. The parameters of the algorithm are: the maximum number of adaptive loops $\bar{\ell}$, the preconditioner $P^{\text{ext}}$ for the first and last loop, the preconditioner $P^{\text{int}}$ for the intermediate loops, the tolerance for the error estimator \text{tol}$^{\eta}$, the tolerance and maximum number of iterations for the intermediate loops  \text{tol}$^{\text{int}}$ and  \text{max\_iter}$^{\text{int}}$, and \text{tol}$^{\text{ext}}$ and  \text{max\_iter}$^{\text{ext}}$ for the first and last loop respectively. 

As an error estimator, we use any error estimator which works directly with an approximate eigenvector, without assuming that the approximate eigenvector satisfies any variational optimality condition (e.g. it needs not be a Ritz vector -- or even close to it -- from a given finite element space, see \cite{carstensen2014adaptive,GrubisicOvall,giani2009convergent}) and Section \ref{res_sec}.

\begin{algorithm}[SA-PINVIT(\text{tol}$^{\eta}$, $\bar\ell$, $P^{\text{ext}}$,  $P^{\text{int}}$, \text{tol}$^{\text{ext}}$,  \text{tol}$^{\text{int}}$,  \text{max\_iter}$^{\text{ext}}$, \text{max\_iter}$^{\text{int}}$)]\label{alg_safem_pinvit} Starting from an initial coarse mesh $\mathcal{T}_1$, and an initial random vector $\mathbf{v}_1$, apply PINVIT($\mathbf{v}_1$, $P_1^{\text{ext}}$,  \text{max\_iter}$^{\text{ext}}$,  \text{tol}$^{\text{ext}}$).

Then, for $\ell=1, \dots, \bar{\ell}-1$ do steps $1.-5.$

\begin{enumerate}[1.]
\item S-PINVIT: Apply PINVIT(${v}^0_{\ell}$, $P_\ell^{\text{int}}$,  \text{max\_iter}$^{\text{int}}$,  \text{tol}$^{\text{int}}$), with initial guess $\mathbf{v}^{0}_{\ell}:= I^{\ell}_{\ell-1}\mathbf{v}_{\ell-1}$, and define ${v}_{\ell}$, $\mu_{\ell}$ as the resulting approximated eigenpair.

\item Estimate: Compute $\eta_T(\mu_{\ell},\mathbf{v}_{\ell})$ for any $T$.
\item Check convergence: If $\eta(\mu_{\ell},\mathbf{v}_{\ell})<tol^{\eta}$ go to Step 6.
\item Mark: Choose set of cells to refine $\mathcal{M}_{\ell}\subset \mathcal{T}_{\ell}$ based on $\eta_T(\mu_{\ell},\mathbf{v}_{\ell})$.
\item Refine: Generate new mesh $\mathcal{T}_{{\ell}+1}$.
\end{enumerate}

\vspace{.5cm}

\noindent 6. \emph{PINVIT}:  Apply PINVIT(${v}^0_{\bar\ell}$, $P_{\bar\ell}^{\text{ext}}$,  \text{max\_iter}$^{\text{ext}}$,  \text{tol}$^{\text{ext}}$), with initial guess $\mathbf{v}^{0}_{\bar\ell} = I^{\bar\ell}_{\bar\ell-1}\mathbf{v}_{\bar\ell-1}$.\\
\emph{Output}: nested sequence of meshes $\mathcal{T}_{{\ell}}$, smoothed approximations $(\mu_{\ell},\mathbf{v}_{{\ell}} )$, estimators $\eta(\mu_{\ell}, \mathbf{v}_{\ell})$ for ${\ell}=1, \dots, \bar{{\ell}}-1$, final problem-adapted approximations $(\mu_{\bar{{\ell}}},\mathbf{v}_{\bar{{\ell}}})$ such that $|\mu_{\bar\ell}-\mu_{\bar\ell-1}|/\mu_{\bar\ell-1} \le \text{tol}^{\text{ext}}$ (if the maximum number of \emph{external} iterations was \emph{not} reached).\\
\end{algorithm}

According to \cite{Saad16}, Algorithm \ref{alg_safem_pinvit} will converge if $\gamma_{P_{\ell},k}\to0$
as $\ell\to\infty$. We can monitor this by monitoring the behavior of the relative residuals $\|A_{\ell} (P_{\ell}^{-1}v - v)\|/\|v\|$. Intuitively, the sequence of vectors (subspaces in the blocked version) $
\mathbf{v}_{\bar{{\ell}}}$ can be seen as a perturbation (with a diminishing perturbation size)
of the Ritz vectors $\mathbf{u}_{\ell}$. Since the perturbation size is decreasing and $\mathbf{v}_{\ell}$
converge to $\mathbf{u}_{\ell}$ we have convergence of perturbed inverse iteration as in \cite{Saad16,Oliveira}.

\section{Numerical validation}\label{sec:numerical_experiments}

In order to validate our findings, we implemented both an A-PINVIT and a SA-PINVIT solver in \texttt{C++} using the \texttt{deal.II} library~\cite{ArndtBangerthDavydov-2021-a, ArndtBangerthBlais-2020-a, SartoriGiulianiBardelloni-2018-a}.

The code is available as opensource on a public repository at \url{https://github.com/luca-heltai/sa-pinvit}. It is based on a modification of the tutorial program \texttt{step-50} of the \texttt{deal.II} library~\cite{Clevenger2020}, and it implements the Laplace and Poisson operators using state-of-the-art matrix-free geometric multigrid preconditioners with local smoothing (see, for example, \cite{janssen2011adaptive, Clevenger2021}). The code is fully parallel, and it is HPC ready. All experiments were run on parallel using message passing interface (MPI) parallelization, on a single node with two CPUS, each featuring a 24-cores Intel Xeon 8160 (SkyLake) running at 2.10 GHz, for a total of 48 cores. All the timings in this section refer to this machine.

\subsection{Fichera corner in two dimensions}\label{subsec:fichera_2d}

We now explore in detail the actual performance of SA-PINVIT in the fully adaptive case, and compare it with pure A-PINVIT.
We start with a classical two dimensional Fichera corner (also known as L-shape domain), and refer to \cite{trefethen2006computed} for reference eigenvalue computations. 

This first test is meant to provide an overview of the behaviour of A-PINVIT and SA-PINVIT in terms of preconditioners, smoothers, and maximum number of inner iterations allowed.

The most sofisticated preconditioner we use is a geometric multigrid (GMG) v-cycle, with a single Jacobi iteration as pre and post smoother. The second type of preconditioner we test is a Chebyshev smoother of variable order (see, for example,~\cite{Varga2009}). In the figures, these are indicated using the notation GMG$^k$(Jacobi$^1$) to identify the application of $k$ geometric multigrid v-cyles, with one application of the Jacobi preconditioner as inner smoother.

Similarly, we indicate with Chebyshev$^k$($d$) the application of $k$ steps of polynomial Chebyshev smoothing of order $d$. 
When the Chebyshev degree is equal to one, this smoother coincides with the one used internally by the GMG preconditioner, i.e., it is a Jacobi smoother. 

Figures~\ref{fig:fichera_2d_deg1_e_00} and ~\ref{fig:fichera_2d_deg1_estimator} show the error in the first eigenvalue computation for bi-linear finite elements, and the difference in the estimator evaluated on the (inexact) intermediate solvers, and on the PINVIT iteration (which is iterated until the iterative tolerance is lower than $10^{-12}$) with different smoothers.

\begin{figure}
  \includegraphics[width=\textwidth]{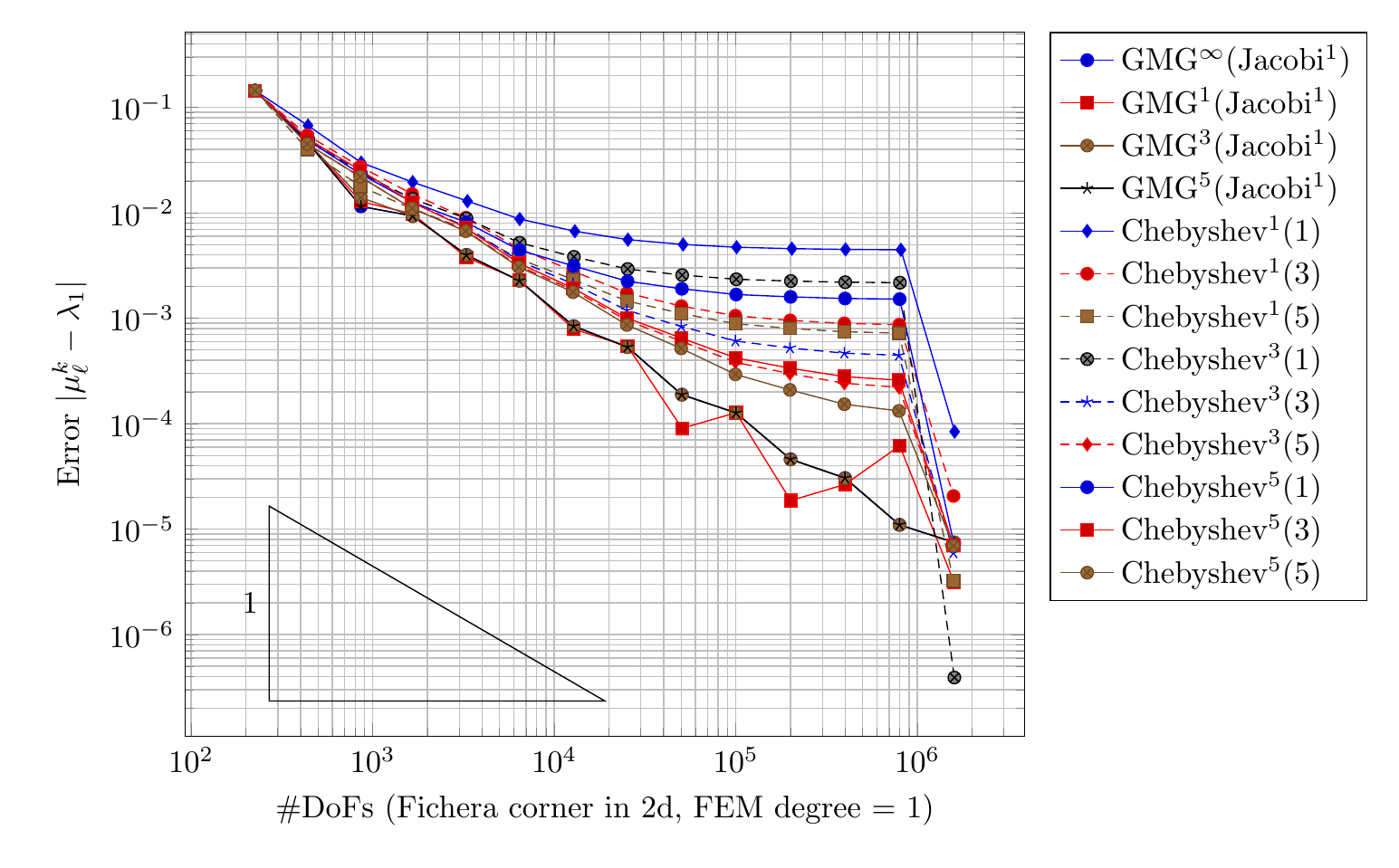}
  \caption{Error in the computation of the lowermost eigenvalue for SA-PINVIT vs A-PINVIT, Fichera corner problem in 2D.
  The comparison shows pure A-PINVIT (using as preconditioner a geometric multigrid v-cycle iteration with one cycle of Jacobi iteration as internal smoother, indicated with GMG$^\infty$(Jacobi$^1$)), where the $\infty$ is there to indicate that we iterate until convergence to a tolerance of $10^{-12}$), and SA-PINVIT based on the application of a fixed number of the same v-cycle algorithm, or a fixed number of smoothing steps using Chebyshev polynomials of order 1, 3, or 5 as preconditioner, indicated with Chebyshev$^k$($d$) where $k$ is the number of iteration steps, and $d$ is the polynomial degree of the Chebyshev smoother.}
  \label{fig:fichera_2d_deg1_e_00}
\end{figure}

\begin{figure}
  \includegraphics[width=\textwidth]{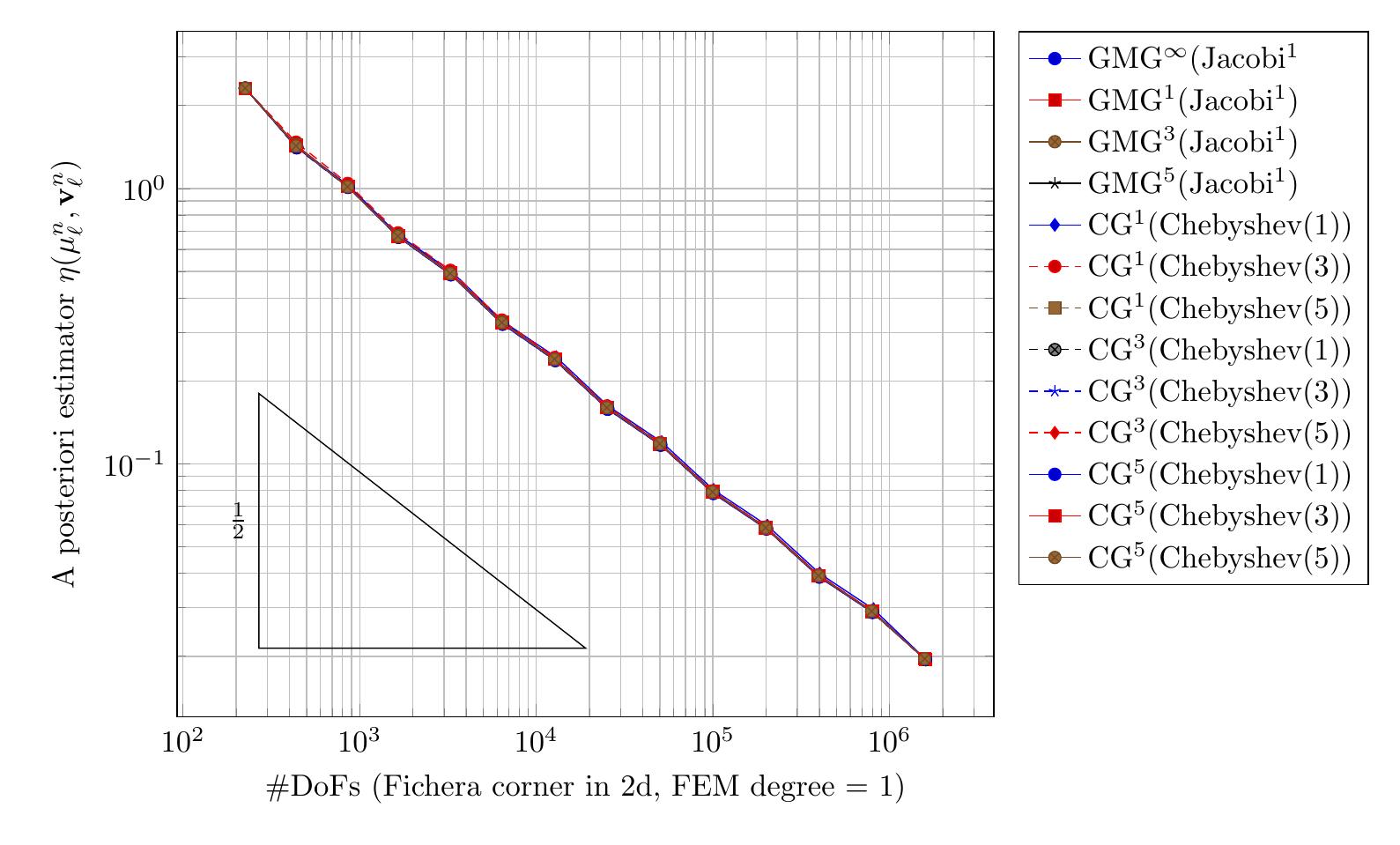}
  \caption{Error estimator for SA-PINVIT vs A-PINVIT, Fichera corner problem in 2D.
  The comparison shows pure A-PINVIT (using as preconditioner a geometric multigrid v-cycle iteration with one cycle of Jacobi iteration as internal smoother, indicated with GMG$^\infty$(Jacobi(1)), where the $\infty$ is there to indicate that we iterate until convergence to a tolerance of $10^{-12}$), and SA-PINVIT based on the application of a fixed number of the same v-cycle algorithm, or a fixed number of conjugate gradient iterations with Chebyshev smoother of degree 1,3, or 5 as preconditioner, indicated with CG$^k$(Chebyshev($d$))) where $k$ is the number of iteration steps, and $d$ is the polynomial degree of the Chebyshev smoother.}
  \label{fig:fichera_2d_deg1_estimator}
\end{figure}

From the figures, it is clear that choosing a better smoother (i.e., using a higher order in the Chebyshev polynomial expansion,  increasing the number of smoothing steps, or using a v-cycle multigrid iteration) decreases the error in the computation of the lowermost eigenvalue in the intermediate stages, but -- in most of the cases -- it has little or no effect on the error that we obtain on the final refinement grid (where full PINVIT is used, with the same configuration as above). 

Just as it happens for the source problem~\cite{mulita2019quasi}, the estimator (and therefore the final mesh pattern) is almost insensitive to the quality of the smoother, and all smoothers used in this set of experiments provide almost the same global estimator (see Figure~\ref{fig:fichera_2d_deg1_estimator}). 

When analyzing the convergence properties of the inexact (smoothed) algorithm, one should concentrate on the concept of the total error in the sense of \cite{StrakosPapez}. The total error of an eigenvalue/vector approximation consists of two parts, the discrete residual measure and the algebraic error,  see Theorem \ref{sep}. We use the discrete residual measure to drive the adaptivity and we control the total energy of the current iterate(s) by smoothing to limit the influence of the algebraic error. When analyzing the discrete residual of a given eigenvector approximation, one sees that it essentially depends on the topology of the mesh (as well as on the properly scaled eigenvector). As such, the convergence of the estimator indicates that the best possible approximation which can be obtained from the current mesh topology is improving by refinement. In view of Theorem \ref{sep}, as well as common sense, it does not imply that the total error converges. Subsequently, a marking strategy based on such an indicator produces a similar refinement as it would have if were we to compute the discrete residual measure from the optimal Galerkin approximation instead of the smoothed prolongation from the previous mesh. The smoothing process primarily ensures that the coefficient in front of the algebraic error does not explode. 

Using a lower quality smoother, however, has a strong impact on the computational cost of the last step, which will take longer to converge if the initial guess is too far away from the exact solution (see Figure~\ref{fig:fichera_2d_deg1_cost}). Choosing a better smoother in the intermediate steps does not have a large benefit on the final solver accuracy, but it does help in decreasing its computational cost.

The analysis of Figures~\ref{fig:fichera_2d_deg1_e_00},~\ref{fig:fichera_2d_deg1_estimator} and~\ref{fig:fichera_2d_deg1_cost} shows that it is necessary to find a balance between accuracy in the intermediate levels, and the overall computational cost. In the tests we performed, the best balance seems to be obtained by applying as a smoother a small fixed number of GMG v-cycle iterations (either two or three). Figure~\ref{fig:mesh_comparison_fichera_2d} shows that the actual mesh refinement patterns generated by A-PINVT and SA-PINVIT with one and three steps of GMG preconditioner respectively are almost identical, at a significant fraction of the total computational cost.

\begin{figure}
  \includegraphics[width=\textwidth]{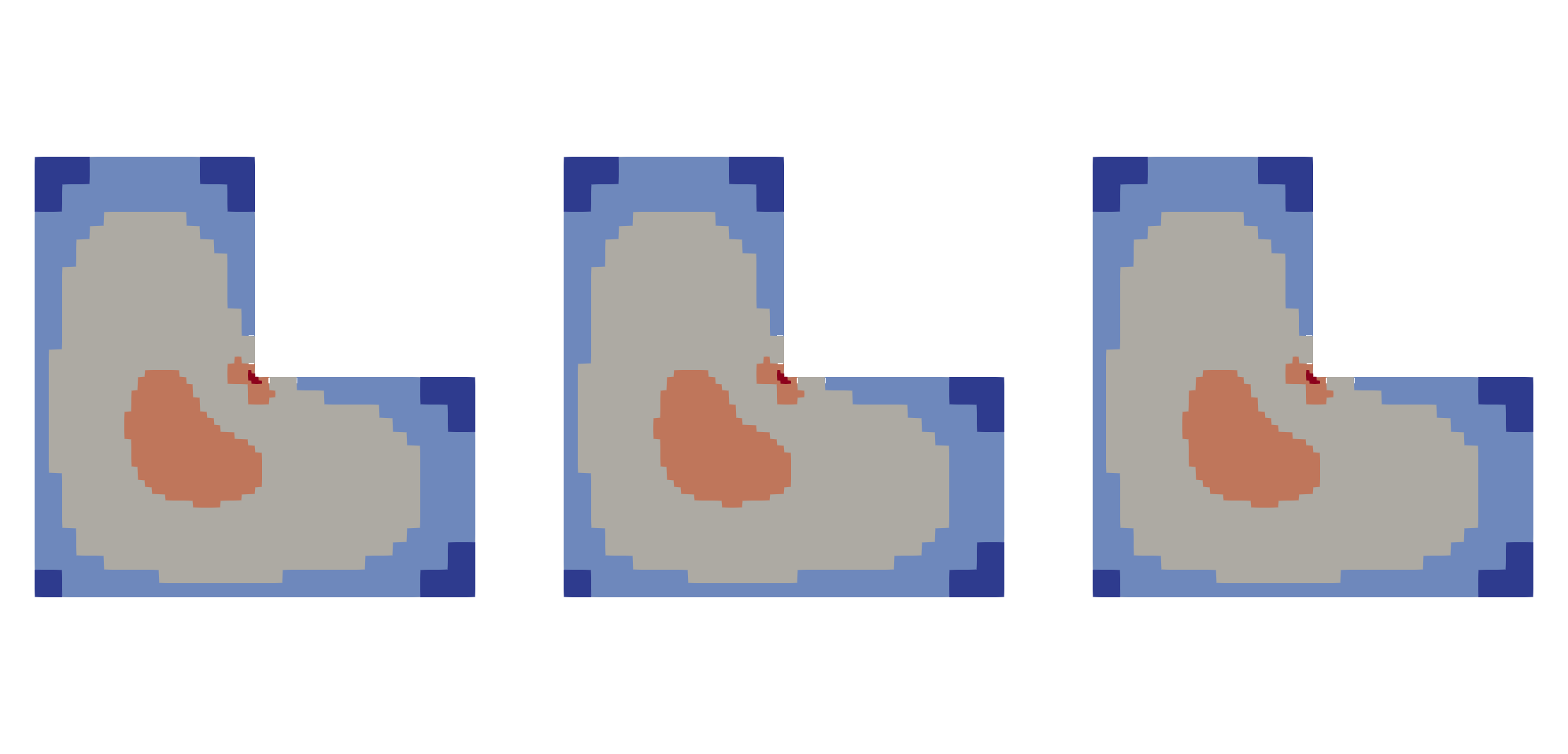}
  \caption{Mesh pattern comparison on the Fichera corner problem in 2D, with ten adaptive refinements. Left: full A-PINVIT (98.457 cells on final level), center: SA-PINVIT with GMG$^1$(Jacobi$^1$) (98.601 cells on final level), right: SA-PINVIT with GMG$^3$(Jacobi$^1$) (98505 cells on final level). The plot shows higher levels of refinement in red, in the intermediate step number four.}
  \label{fig:mesh_comparison_fichera_2d}
\end{figure}

\begin{figure}
  \includegraphics[width=\textwidth]{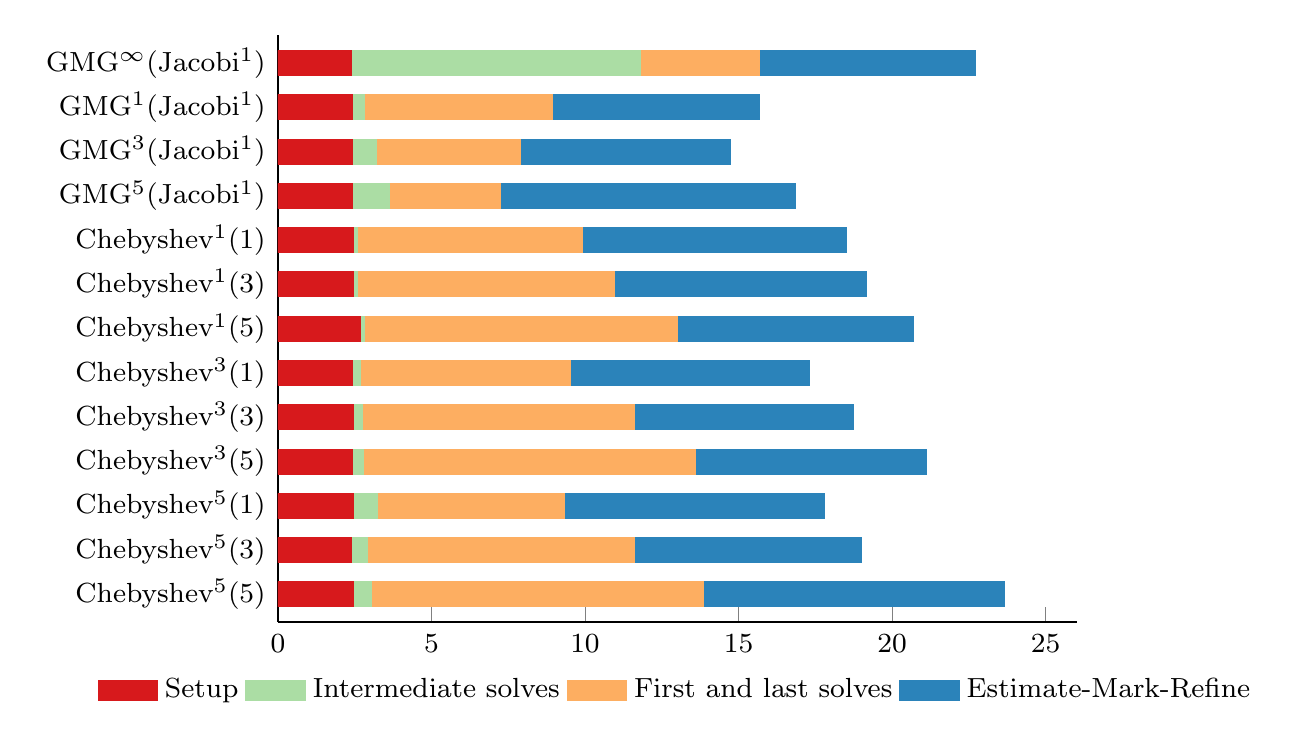}
  \caption{Computational cost of A-PINVIT and SA-PINVIT for the Fichera corner problem in 2D. The timing is in seconds, and the dimension of the problem goes to $O(10^{6})$ degrees of freedom.}
  \label{fig:fichera_2d_deg1_cost}
\end{figure}

Finally, in Figure~\ref{fig:fichera_2d_deg1_all_eigenvalues} we provide a convergence plot for the first six eigenvalues of the L-shaped domain when using full A-PINVIT or SA-PINVIT based on the application of two GMG v-cycle steps as smoother. The plot shows optimal (linear) convergence of the eigenvalues w.r.t. the global number of degrees of freedom, in both A-PINVIT and SA-PINVIT, also in the intermediate steps, showing that this choice of smoother is very effective in providing a good approximation of the lowest part of the spectrum with very few iterations.

\begin{figure}
  \includegraphics[width=\textwidth]{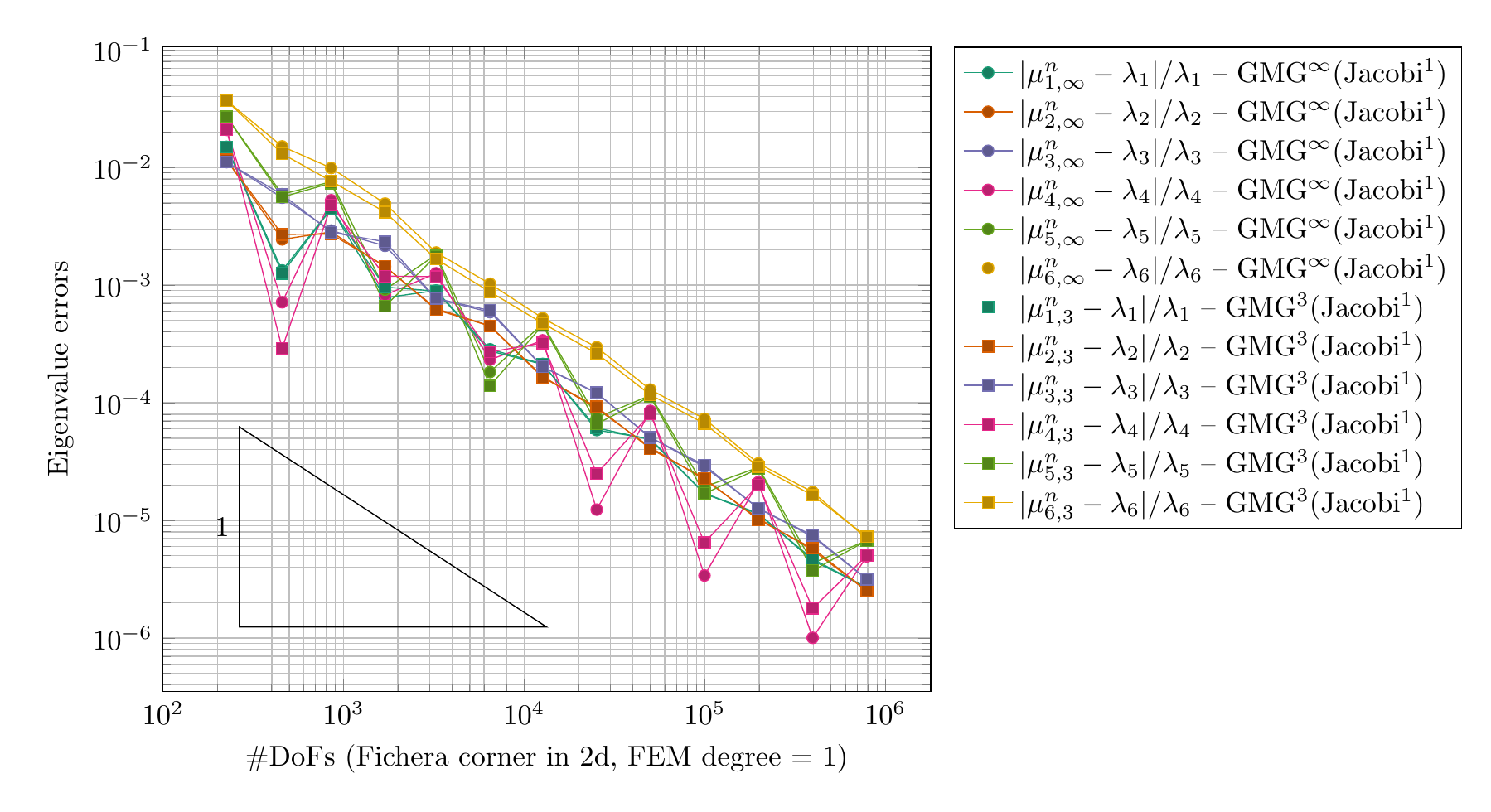}
  \caption{Error in the first six lowermost eigevnalues for SA-PINVIT vs A-PINVIT, Fichera corner problem in 2D. The estimator is computed w.r.t. the first eigenvalue and eigenvector pair.
  The comparison shows pure A-PINVIT (using as preconditioner a geometric multigrid v-cycle iteration with two cycles of Jacobi iteration as internal smoother, indicated with GMG$^\infty$(Jacobi(1)), where the $\infty$ is there to indicate that we iterate until convergence to a tolerance of $10^{-12}$), and SA-PINVIT based on the application of three iterations of the same v-cycle algorithm.}
  \label{fig:fichera_2d_deg1_all_eigenvalues}
\end{figure}

\subsection{Dumbbell in two dimensions} \label{subsec:dumbbell_2d}

A more difficult test case is given by a dumbbell domain in two dimensions. We take two copies of the same rectangle connected by the small bridge (see~\cite{trefethen2006computed} for the detailed description of the domain). If we were just to have two rectangles, we would have the multiplicities of the eigenvalues of the Laplace operator on the square doubled. However, due to the small connection the eigenvalues degenerate in the cluster of eigenvalues of the same joint multiplicity. We consider $6$ lowermost eigenvalues since this cluster is relatively tight and well separated from the rest of the spectrum. As reference eigenvalues we use highly accurate eigenvalues computed by the $hp$ adapted DG method \cite{GIANI2015604}. For an alternative method to compute highly accurate eigenvalues of polygonal regions see \cite{trefethen2006computed}.

In Figure~\ref{fig:dumbbell_2d_deg1_all_eigenvalues} we plot the error in the six lowermost eigenvalues (shown in Figure~\ref{fig:eigenfunctions_dumbbell_mesh}), when the error estimator is computed only with the first eigenpair. We show the difference between A-PINVIT and SA-PINVIT, where the same colors indicate the same eigenvalue, while circle markers identify A-PINVIT and squares ones identify SA-PINVIT.

\begin{figure}
  \includegraphics[width=\textwidth]{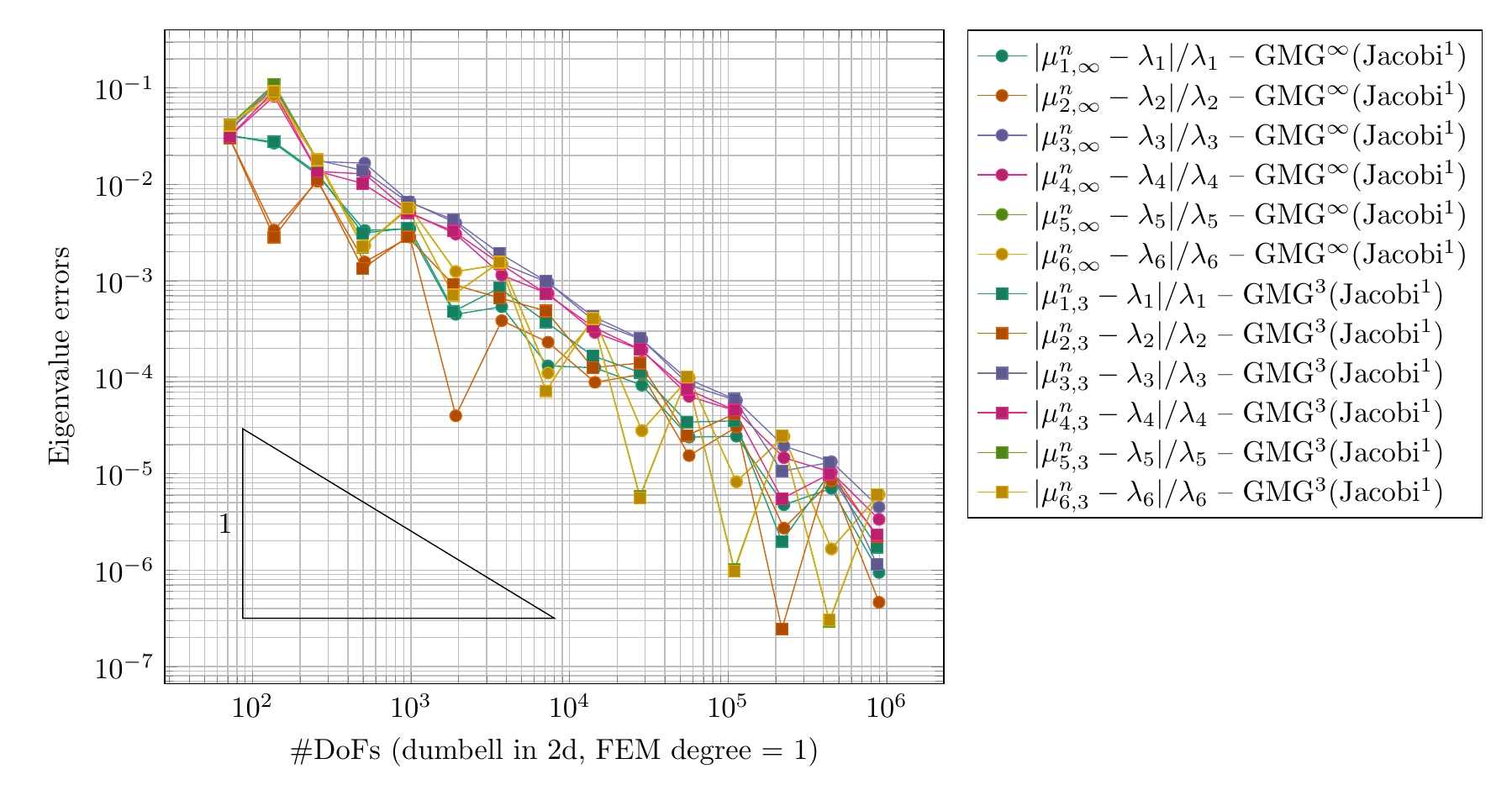}
  \caption{Error in the first six lowermost eigevnalues for SA-PINVIT vs A-PINVIT, dumbbell problem in 2D. The estimator is computed w.r.t. the first eigenvalue and eigenvector pair.
  The comparison shows pure A-PINVIT (using as preconditioner a geometric multigrid v-cycle iteration with one cycle of Jacobi iteration as internal smoother, indicated with GMG$^\infty$(Jacobi(1)), where the $\infty$ is there to indicate that we iterate until convergence to a tolerance of $10^{-12}$), and SA-PINVIT based on the application of three iterations of the same v-cycle algorithm.}
  \label{fig:dumbbell_2d_deg1_all_eigenvalues}
\end{figure}

In this case, the difference in the mesh sequences between A-PINVIT and SA-PINVIT using one GMG v-cycle as smoothing step is visible (see top and center plots in Figure~\ref{fig:mesh_comparison_dumbbell_2d}), even though the convergence of the eigenvalues is essentially the same. 

Notice how introducing local refinement breaks the symmetry of the eigenvalue problem. A mesh generated by a global marking strategy does not need to be symmetric (unless the marking strategy is designe to enforce the symetry structurally) and so there is no reason for a given eigenvector approximation to be mesh symmetric. This effect is particularly explicit in the case of a double eigenvalue and an iterative method which uses random starts or restarts, such as Arnoldi or Krylov-Schur. However, the subspace spanned by a collection of eigenvectors associated to a multiple eigenvalue or a tight cluster of eigenvalues will be stable with respects to the perturbations introduced by refinement. Furthermore the subspaces spannd by the approximate eigenvectors and the true egenvectors will be  close in terms of the subspace angles. Note that small subspace angle does not imply any mesh symmetry in the preasymptotic regime. In Figure \ref{fig:eigenfunctions_dumbbell_mesh} we present highly accurate eigenvector approximations computed by our algorithm from the finite element space of $10^{6}$ degrees of freedom. On can observe the nodal lines marking the expected symmetries of the dumbbell. This indicates that the symmetries will be achieved asymptotically. 

The difference in the grids is still there, but not noticeable, between PINVIT and SA-PINVIT when using three steps of GMG v-cycles. The computational cost for this test case is presented in Figure~\ref{fig:dumbbell_2d_deg1_cost}, where two GMG v-cycles are used for SA-PINVIT, showing a speedup of about three in the total computational cost, where the largest saving is clearly in the intermediate solution steps.

\begin{figure}
  \centering
  \includegraphics[width=.6\textwidth]{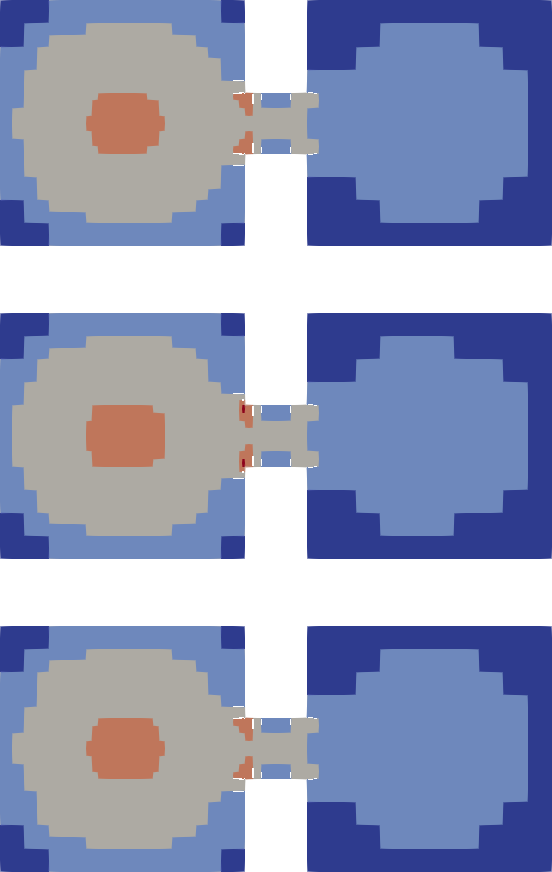}
  \caption{Mesh pattern comparison on the dumbbell problem in 2D, with ten adaptive refinements. Top: full A-PINVIT (27.414 cells on final level), center: SA-PINVIT with GMG$^1$(Jacobi$^1$) (27.972 cells on final level), bottom: SA-PINVIT with GMG$^3$(Jacobi$^1$) (27.318 cells on final level). The plot shows higher levels of refinement in red, in the intermediate step number five.}
  \label{fig:mesh_comparison_dumbbell_2d}
\end{figure}

\begin{figure}
  \includegraphics[width=\textwidth]{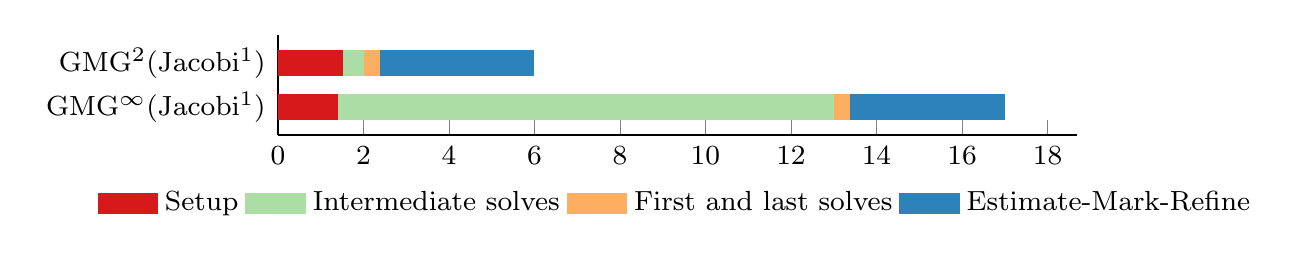}
  \caption{Computational cost of A-PINVIT and SA-PINVIT for the dumbbell problem in 2D. The timing is in seconds, and the dimension of the problem goes to $O(10^{6})$ degrees of freedom.}
  \label{fig:dumbbell_2d_deg1_cost}
\end{figure}

\subsection{Fichera corner in three dimensions}\label{subsec:fichera_3d}

A more challenging test case is given by a three-dimensional version of the Fichera corner, where we extrude the two dimensional L-shaped domain along the $z-$direction. Similarly to what we did in the two dimensional case, we apply the error estimator on the first eigenfunction/eigenvalue pair, and we plot the error in the first six lowermost eigenalues in Figure~\ref{fig:fichera_3d_deg1_all_eigenvalues}. In the three-dimensional case, the convergence of the eigenvalues is much more oscillating, and while the computational saving in the intermediate stage is similar to the two dimensional case, the overall computational cost in this case is dominated by the ESTIMATE-MARK-REFINE steps, rather than by the solution stages (see Figure~\ref{fig:fichera_3d_deg1_cost}).

\begin{figure}
  \includegraphics[width=\textwidth]{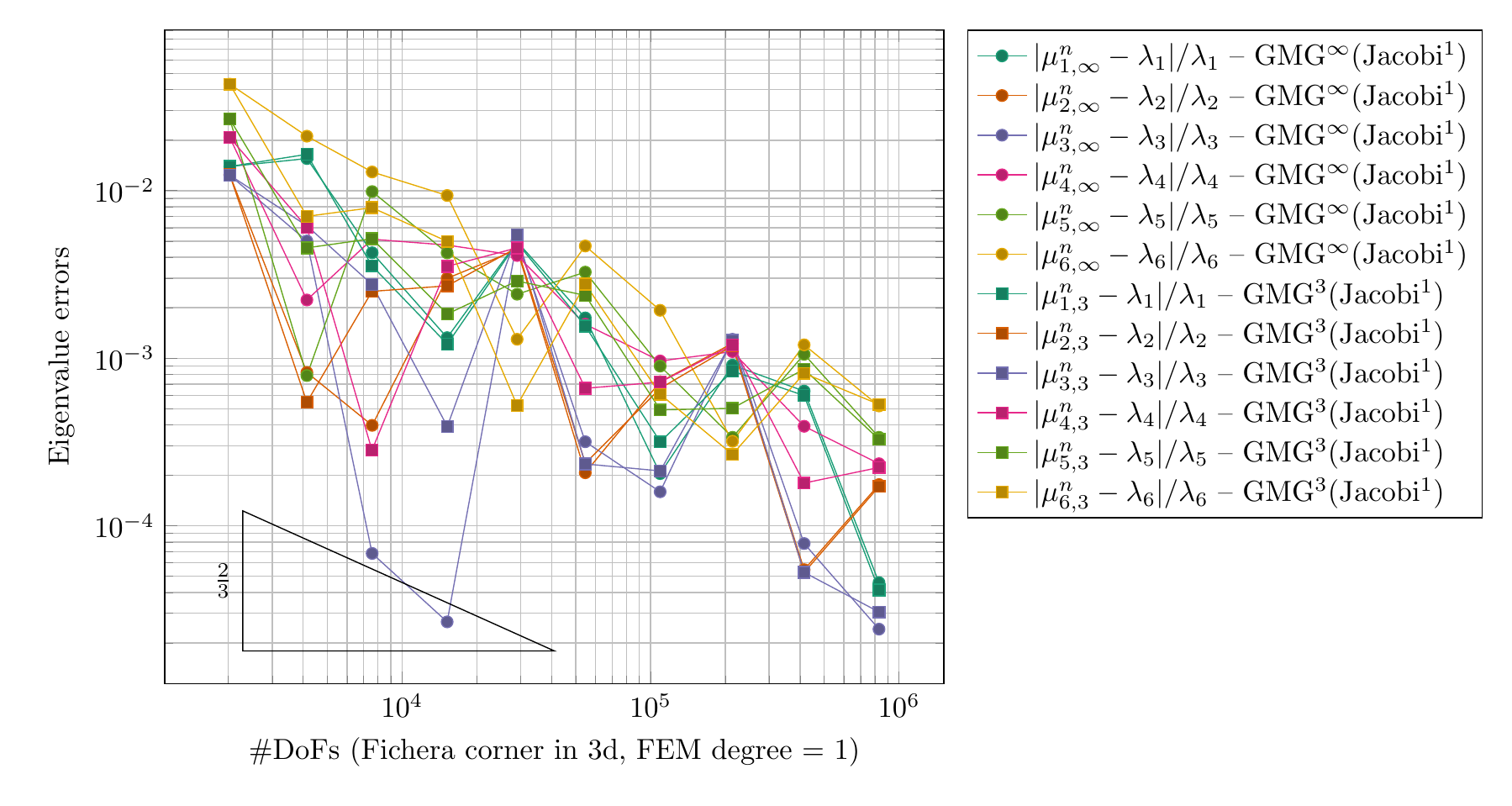}
  \caption{Error in the first six lowermost eigevnalues for SA-PINVIT vs A-PINVIT, Fichera corner problem in 3D. The estimator is computed w.r.t. the first eigenvalue and eigenvector pair.
  The comparison shows pure A-PINVIT (using as preconditioner a geometric multigrid v-cycle iteration with one cycle of Jacobi iteration as internal smoother, indicated with GMG$^\infty$(Jacobi(1)), where the $\infty$ is there to indicate that we iterate until convergence to a tolerance of $10^{-12}$), and SA-PINVIT based on the application of three iterations of the same v-cycle algorithm.}
  \label{fig:fichera_3d_deg1_all_eigenvalues}
\end{figure}

\begin{figure}
  \includegraphics[width=\textwidth]{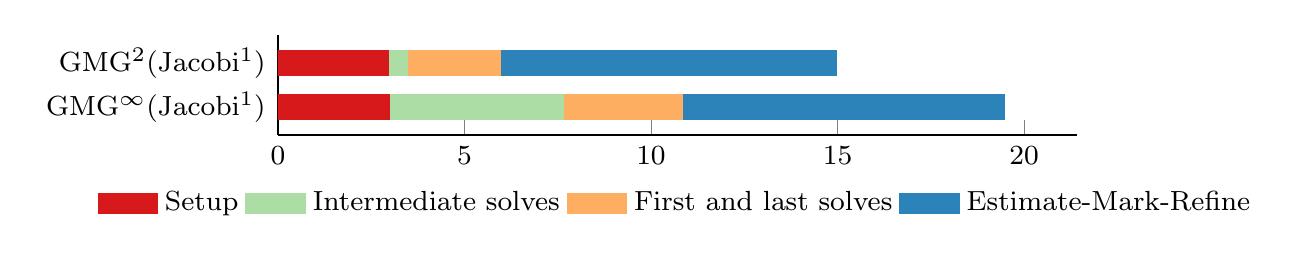}
  \caption{Computational cost of A-PINVIT and SA-PINVIT for the Fichera corner problem in 3D. The timing is in seconds, and the dimension of the problem goes to $O(10^{6})$ degrees of freedom.}
  \label{fig:fichera_3d_deg1_cost}
\end{figure}

A comparison of the mesh sequences in the three-dimensional case is difficult to visualize. We show a snapshot of the grid resulting from four steps of SA-PINVIT iteration in Figure~\ref{fig:fichera_3d_deg1_grid}.

\begin{figure}
  \centering
  \includegraphics[width=.8\textwidth]{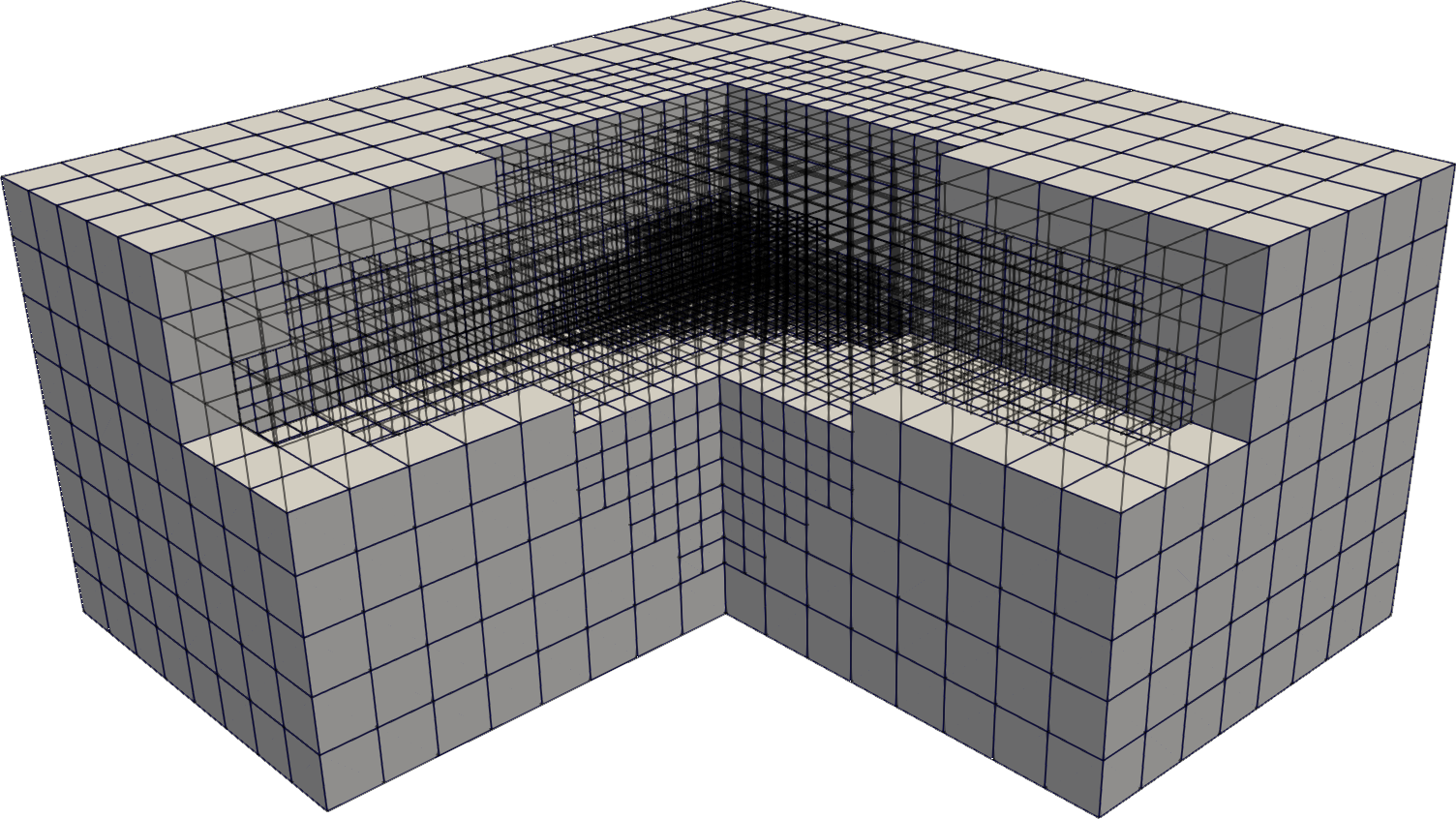}
  \caption{Local refinement after four cycles for SA-PINVIT iteration for the Fichera corner problem in 3D. Part of the grid is left as a transparent wireframe to show the internal local refinements.}
  \label{fig:fichera_3d_deg1_grid}
\end{figure}

\subsection{Dumbbell in three dimensions}\label{subsec:dumbbell_3d}

Similarly to what we did for the Fichera corner, we study now the three dimensional version of the dumbbell problem presented in Section~\ref{subsec:dumbbell_2d}. In this case, the convergence is even more oscillatory (Figure~\ref{fig:dumbbell_3d_deg1_all_eigenvalues}). This is the most challenging problem for SA-PINVIT, since local refinement breaks the symmetry of the problem, and leads to a mixing of the first two eigenpairs along the refinement paths. 

This is a known challenge of the design of cluster robust mesh refinement strategies. Namely, in the case of multiple or tightly clustered eigenvalues, one needs to devise basis independent error estimators. One of the consequences of this feature of the eigenvalue problem in the presence of multiple eigenvalues is that when using algebraic algorithms based on random starts and restarts (Arnoldi, Krylov Schur, ...) one might get a different basis of the eigenspace associated with the same multiple eigenvalue in every new run of an algorithm on applied on the same input matrices. Instead an error estimator should depend on the subspace spanned by approximate eigenvectors, see \cite{GrubisicOvall,BGO}. However, such marking strategies flag many elements for refinement and as such would produce overly fine meshes. This goes directly against the very principle behind the design of the current algorithm. Therefore we have opted to let the multiplicity, or near multiplicity, be resolved asymptotically rather than resorting to the use of a safer but more aggressive subspace dependent mesh refinement strategy. On the dumbbell example, we clearly see an echo of the challenge faced by the estimator in the presence of a tight cluster or eigenvalues. This is precisely the reason for choosing the dumbbell as a benchmark example.

This is evident from the local refinement pattern, as shown in Figure~\ref{fig:dumbbell_3d_deg1_grid}. A possible explanation for the large oscillations in the convergence of the eigenvalues in Figure~\ref{fig:dumbbell_3d_deg1_all_eigenvalues} is given by the fact that, while advancing with the local refinement, the first two eigen pairs exchange their relative order, and the estimator is applied to a different eigenfunction w.r.t. what was done in the previous refinement cycle, resulting in a loss of performance on the desired eigenfunction.

\begin{figure}
  \includegraphics[width=\textwidth]{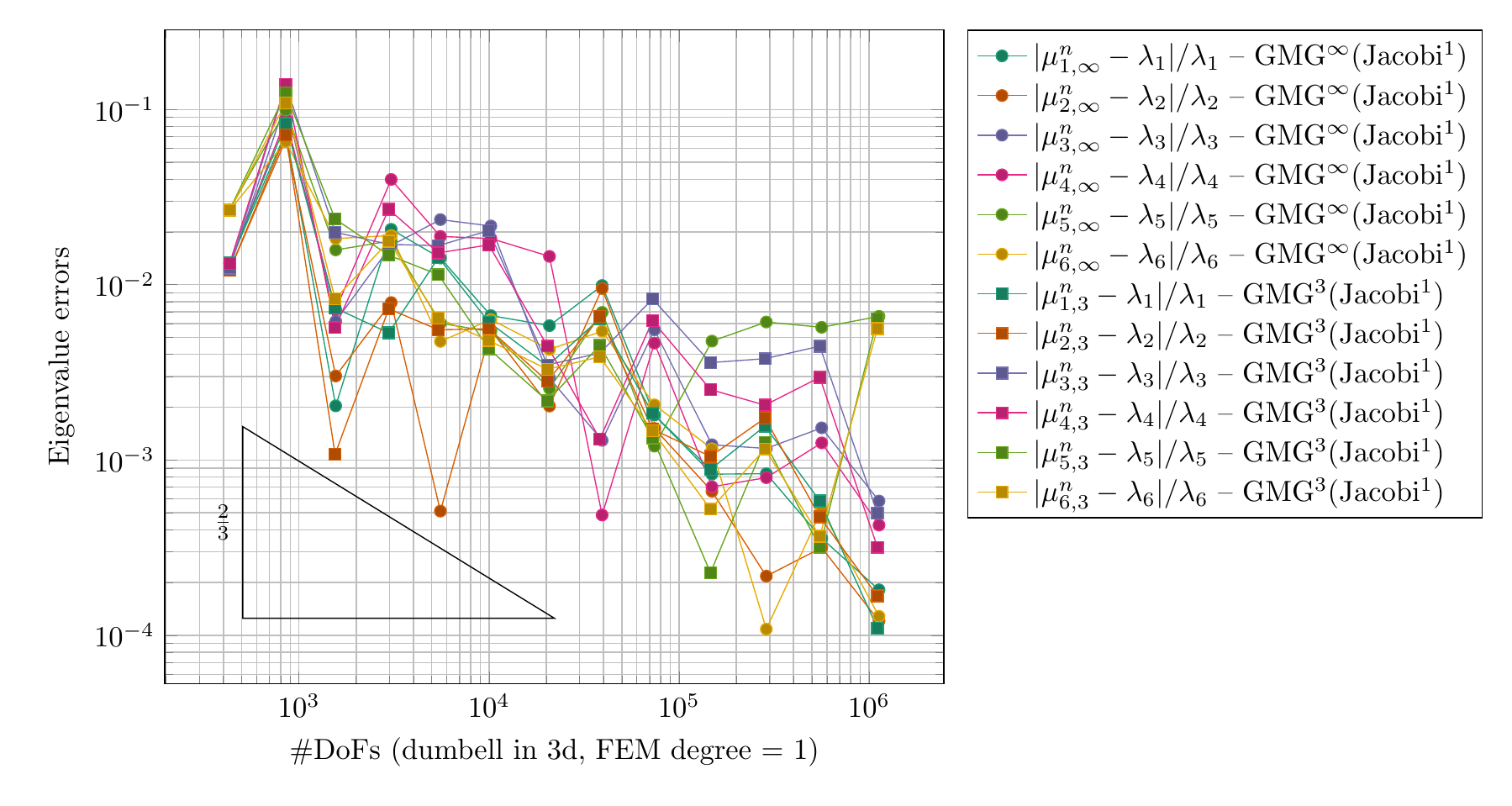}
  \caption{Error in the first six lowermost eigevnalues for SA-PINVIT vs A-PINVIT, dumbbell problem in 3D, finite element degree equal to 1. The estimator is computed w.r.t. the first eigenvalue and eigenvector pair.
  The comparison shows pure A-PINVIT (using as preconditioner a geometric multigrid v-cycle iteration with one cycle of Jacobi iteration as internal smoother, indicated with GMG$^\infty$(Jacobi(1)), where the $\infty$ is there to indicate that we iterate until convergence to a tolerance of $10^{-12}$), and SA-PINVIT based on the application of three iterations of the same v-cycle algorithm.}
  \label{fig:dumbbell_3d_deg1_all_eigenvalues}
\end{figure}

Such difficulty is also evident in the overall computational cost presented in Figure~\ref{fig:dumbbell_3d_deg1_cost}, where it is clear that the last solution does not benefit at all from the final prolongation of the solution from the intermediate SA-PINVIT steps. This numerical experiments exposes an area where it is necessary to improve the SA-PINVIT algorithm: very close but distinct eigenvalues may interchange during local refinement, and it requires that the whole block of close eigenvalues is taken into account when implementing a refinement strategy, or the benefit of SA-PINVIT may get lost due to an excessive cost in the last (finest) iteration, as Figure~\ref{fig:dumbbell_3d_deg1_cost} suggests.

\begin{figure}
  \includegraphics[width=\textwidth]{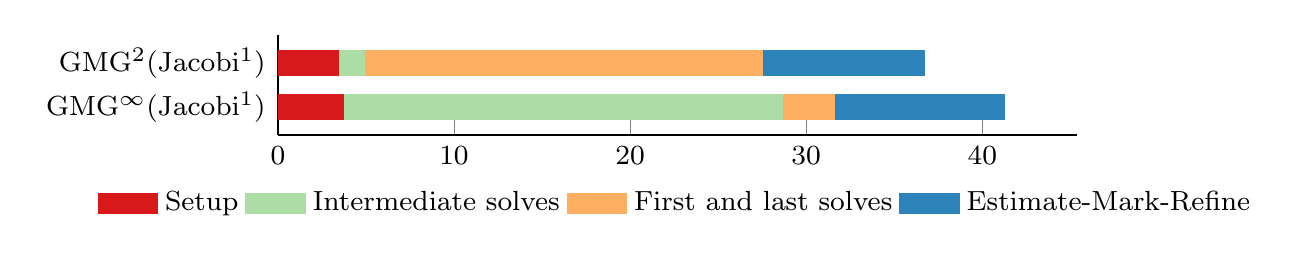}
  \caption{Computational cost of A-PINVIT and SA-PINVIT for the dumbbell problem in 3D, with finite element degree equal to 1. The timing is in seconds, and the dimension of the problem goes to $O(10^{6})$ degrees of freedom.}
  \label{fig:dumbbell_3d_deg1_cost}
\end{figure}

\begin{figure}
  \centering
  \includegraphics[width=.8\textwidth]{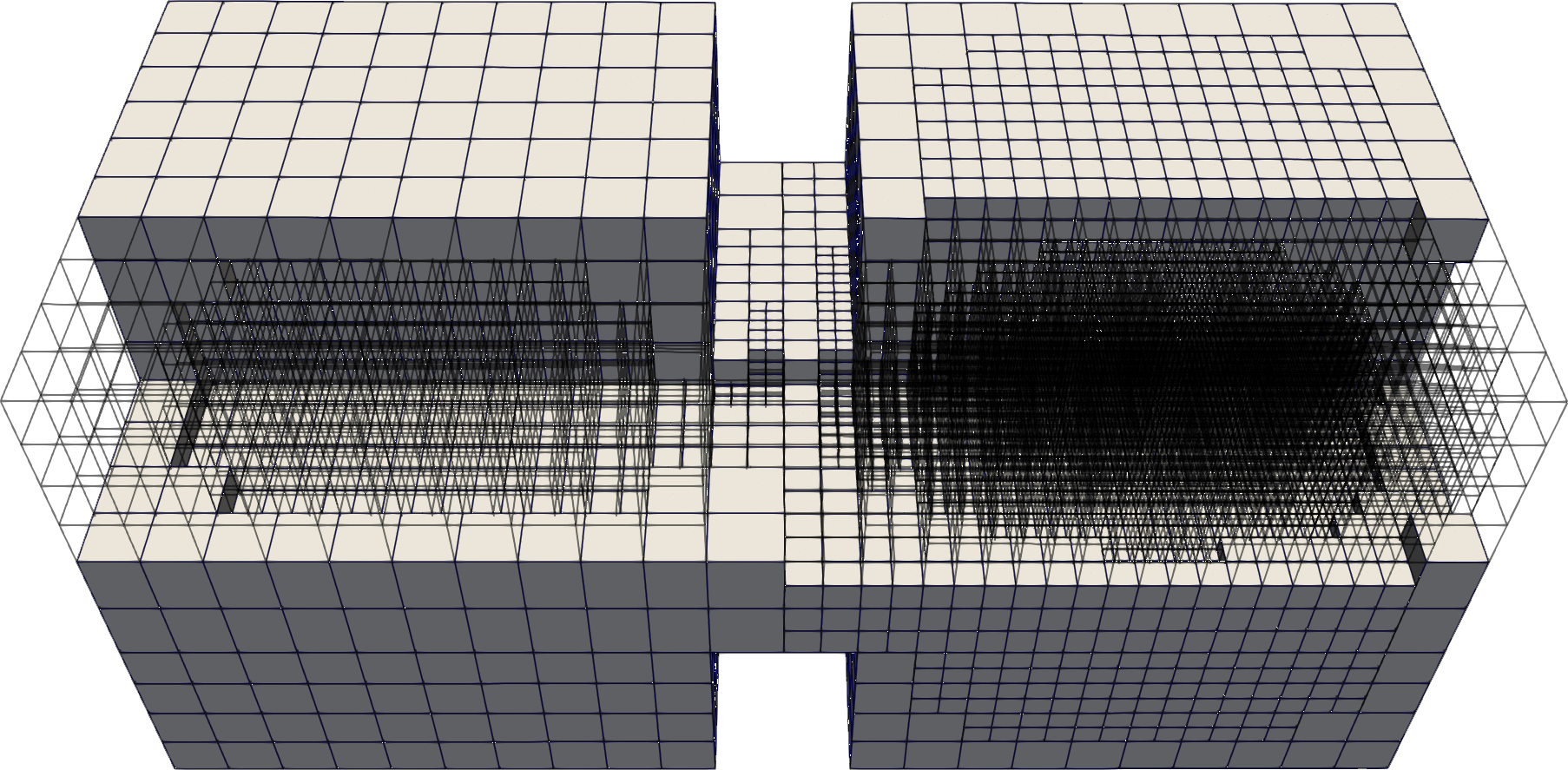}
  \caption{Local refinement after four cycles for SA-PINVIT iteration in the dumbbell problem in 3D, with finite element degree equal to 1. Part of the grid is left as a transparent wireframe to show the internal local refinements.}
  \label{fig:dumbbell_3d_deg1_grid}
\end{figure}

\section{Concluding remarks}\label{sec:summary}
In this paper we present an eigensolver designed for AFEM spaces aimed to increase the efficiency of an algorithm when applied to a sequence of locally refined meshes. The method which we proposed is motivated by the work on inexact inverse iteration solvers in numerical linear algebra and recent work in the perturbed iterative methods and on smoothed-AFEM for source problems. Following work on inexact inverse iteration we see that such perturbed hybrid iterative methods converge under very mild assumptions on the control of the residuals of the intermediate inexact solvers (convergence of residual norms to zero, see \cite{Saad16}). 

The performed experiments demonstrate that we achieve quasi-optimal convergence rates. We have explicitly shown the intricate relationship between various stages of the adaptive process (Setup, Intermediate solves, First and last solves, Estimate-Mark-Refine). The experiments were designed to push the boundary and ascertain the influence of the spectral separation on various stages of the process (these effects were more pronounced for problems with broken symmetries and in 3D). The change of the relative significance of the cost of various stages has been compared for, in this respect, challenging 2D and 3D experiments.

\section*{Acknowledgements} L.G. was supported by the Croatian Science Foundation grant HRZZ IP-2019-04-6268. L.G. is also thankful to the hospitality of the research visit to Scuola Internazionale Superiore di Studi Avanzati where the work has started. O.M. is thankful to the University of Zagreb for the hospitality during her collaborative research visit there. L.H. was partially supported by the National Research Projects (PRIN  2017) `'`Numerical Analysis for Full and Reduced Order Methods for the efficient and accurate solution of complex systems governed by Partial Differential Equations'', funded by the Italian Ministry of Education, University, and Research.

\printbibliography
\end{document}